\documentclass[a4paper,11pt,draft]{amsart}

\usepackage[margin=18mm]{geometry}
\usepackage{amssymb, amsmath, amsthm, amscd}

\usepackage[T1]{fontenc}
\usepackage{eucal,mathrsfs,dsfont}%gives nicer set-names. Use \mathds instead of \mathds
\usepackage{color}

\numberwithin{equation}{section}

\newcommand{\be}{\begin{eqnarray}}
\newcommand{\ee}{\end{eqnarray}}
\newcommand{\ce}{\begin{eqnarray*}}
\newcommand{\de}{\end{eqnarray*}}
\newtheorem{theorem}{Theorem}[section]
\newtheorem{lemma}[theorem]{Lemma}
\newtheorem{remark}[theorem]{Remark}
\newtheorem{definition}[theorem]{Definition}
\newtheorem{proposition}[theorem]{Proposition}
\newtheorem{Examples}[theorem]{Example}
\newtheorem{corollary}[theorem]{Corollary}

\def\Ee{\mathbb{E}}
\def\Pp{\mathbb{P}}

\def\om{\omega}

\def\[{{\Big[}}
\def\]{{\Big]}}
\def\<{{\langle}}
\def\>{{\rangle}}
\def\({{\Big(}}
\def\){{\Big)}}

\def\bx{{\mathbf{x}}}

\def\min{{\mathord{{\rm min}}}}

\def\={&\!\!=\!\!&}

\def\cE{{\mathcal E}}
\def\cF{{\mathcal F}}

\def\bP{{\mathbf P}}

\def\bE{{\mathbf E}}
\def\1{{\mathbf{1}}}

\def\geq{\geqslant}
\def\leq{\leqslant}
\def\ge{\geqslant}
\def\le{\leqslant}

\def\om{\omega}

\def\[{{\Big[}}
\def\]{{\Big]}}
\def\<{{\langle}}
\def\>{{\rangle}}
\def\({{\Big(}}
\def\){{\Big)}}

\def\bx{{\mathbf{x}}}

\def\I{\mathbf 1}

\def\min{{\mathord{{\rm min}}}}

\def\={&\!\!=\!\!&}
\def\bt{\begin{theorem}}
\def\et{\end{theorem}}
\def\bl{\begin{lemma}}
\def\el{\end{lemma}}
\def\br{\begin{remark}}
\def\er{\end{remark}}
\def\bx{\begin{Examples}}
\def\ex{\end{Examples}}
\def\bd{\begin{definition}}
\def\ed{\end{definition}}
\def\bp{\begin{proposition}}
\def\ep{\end{proposition}}
\def\bc{\begin{corollary}}
\def\ec{\end{corollary}}

\def\geq{\geqslant}
\def\leq{\leqslant}
\def\ge{\geqslant}
\def\le{\leqslant}

\def\bP{{\mathbf P}}

 \def\R{\mathbb R}
 \def\R{\mathbb R}    
\def\N{\mathbb N}  
   
\def\<{\langle} \def\>{\rangle}

\allowdisplaybreaks

\definecolor{darkergreen}{rgb}{0.0, 0.5, 0.0}

\begin{document}

\allowdisplaybreaks
\title[Heat kernel estimates for Brox's diffusion]
{\bfseries Quenched and annealed heat kernel estimates for Brox's diffusion
}

\author{Xin Chen\qquad Jian Wang}

\date{}
\thanks{\emph{X.\ Chen:}
   School of Mathematical Sciences, Shanghai Jiao Tong University, 200240 Shanghai, P.R. China. \texttt{chenxin217@sjtu.edu.cn}}
     \thanks{\emph{J.\ Wang:}
    School of Mathematics and Statistics \& Key Laboratory of Analytical Mathematics and Applications (Ministry of Education) \& Fujian Provincial Key Laboratory
of Statistics and Artificial Intelligence, Fujian Normal University, 350007 Fuzhou, P.R. China. \texttt{jianwang@fjnu.edu.cn}}
\maketitle
\begin{abstract}
Brox's diffusion is a typical one-dimensional singular diffusion, which was introduced by Brox (1986) as a continuous analogue of Sinai's random walk. In this paper, we
will establish quenched heat kernel estimates for short time and annealed heat kernel estimates for large time of Brox's diffusion. The proofs are based on Brox's construction via the scale-transformation and the time-change arguments as well as the theory of resistance forms for
symmetric strongly recurrent Markov processes. We emphasize that, since the reference measure of Brox's diffusion does not satisfy the so-called volume doubling conditions neither for the small scale nor the large scale,
the
existing
methods for heat kernel estimates of diffusions in ergodic media do not work, and new
techniques will be introduced to establish both quenched and annealed heat kernel estimates of Brox's diffusions, which
take into account different oscillation properties for one-dimensional Brownian motion in random environments.

\medskip

\noindent\textbf{Keywords:} Brox's diffusion; heat kernel estimate; scale-transformation;  time-change; resistance form

\medskip

\noindent \textbf{MSC 2010:} 60G51; 60G52; 60J25; 60J75.
\end{abstract}

\section{Introduction and main results}
\subsection{Background}
Let $\Omega:=C(\R;\R)$ be the space of real-valued continuous functions defined on $\R$ and vanishing at the origin, and let $\omega$ denote an element in $\Omega$. Let $P$ be the Wiener measure on $\Omega$; namely, the probability measure on $\Omega$ such that $\{\{\omega(x)\}_{x\ge0}, P\}$ and $\{\{\omega(x)\}_{x\le 0}, P\}$ are independent one-dimensional standard Brownian motions. Given a sample function $W\in \Omega$ (that is, $W(x,\omega)=\omega(x)$ for all $x\in \R$), we consider the following one-dimensional stochastic differential equation (SDE)
\begin{equation}\label{e0}
dX(t)=-\frac{1}{2}\dot{W}(X(t))\,dt+d\beta(t),
\end{equation}
where $\{\beta(t)\}_{t\ge0}$ is a one-dimensional standard Brownian motion that is independent of $\{W(x)\}_{x\in \R}$, and $\dot{W}(x)$ denotes the formal derivative of $W(x)$. (Note that, since $W$ is not differentiable, the SDE \eqref{e0}
 can not be solved in the classical sense.) The SDE \eqref{e0} was first introduced in \cite{Brox} by Brox  as a continuous analogue of Sinai's random walk \cite{Si},
 one of whose motivations is
 to study the interaction  between $\{\beta(t)\}_{t\ge0}$ and $\{W(x)\}_{x\in \R}$
in the large scale of time and space.
 For a fixed sample  $\omega\in \Omega$, let $\bP_\omega^x$ be the probability measure on $C([0,\infty))$ induced by the Brox diffusion $X:=\{X(t)\}_{t\ge0}$ with the starting point $X(0)=x\in \R$ and the
 fixed environments
 $\{W(x,\omega)\}_{x\in \R}=\{\omega(x)\}_{x\in \R}$. That is, $\bP_\omega^x$ denotes the quenched probability of the Brox diffusion $X$ when the randomness induced by the random potential $\{W(x,\omega)\}_{x\in \R}=\{\omega(x)\}_{x\in \R}$ is fixed. While $\{W(x)\}_{x\in \R}$ is random, the process $\{X(t)\}_{t\ge0}$ also is defined on the probability space
$(\Omega\times C([0,\infty)), \Pp^{x})$, where $\Pp^{x}:=P(d\omega)\bP_\omega^{x}$ represents the annealed  probability for the Brox diffusion $X$ induced by the
environments
$\{W(x)\}_{x\in \R}$.

As mentioned before, due to the singularity of the drift $\dot{W}(x)$, the SDE \eqref{e0} above can not be solved by the standard theory for
neither the strong solution nor the weak solution of SDEs. Actually,
 the argument
 in Brox \cite[Section 1]{Brox}
 is based on the time and space transformations as in the It\^{o}-McKean construction of Feller-diffusion process.
In details, the Brox diffusion $X$ can be viewed as a Feller-diffusion process on $\R$ with the generator of Feller's canonical form
$$\frac{e^{ W(x)}}{2}\frac{d}{d x}\left( {e^{-W(x)}}\frac{d}{d x}\right).$$
Through
the It\^{o}-McKean construction of Feller-diffusion process, which applies the scale-transformation and the time-change
to a Brownian motion,
the Brox diffusion $X$ can be explicitly given by
\begin{equation}\label{brox}
X(t)=S^{-1}(B(T^{-1}(t))),\quad t\ge0
\end{equation} with
$$T(t)=\int_0^t\exp(-2W(S^{-1}(B(s))))\,ds,\quad S(x)=\int_0^x e^{W(z)}\,dz,$$ where $B:=\{B(t)\}_{t\ge 0}$ is  a one-dimensional standard Brownian motion starting from the origin on some probability space. As we will see, Brox's construction is crucial for the
arguments in our paper. Later,
through the expression of Brownian local time,
it was proved in \cite[Theorems 2.4 and 2.5]{HLM} that for any Brownian
motion $B$, independent of $W:=\{W(x)\}_{x\in \R}$, the representation \eqref{brox}
is a weak solution to the
SDE \eqref{e0}; and that for any given Brownian motion $\{\beta(t)\}_{t\ge0}$ we
can find a special Brownian motion
$B,$ independent of $W$, such that the representation \eqref{brox}
is a unique strong solution
to the
SDE \eqref{e0}. We also want to remark that another possible way of solving \eqref{e0} rigorously is
based on the
paracontrolled theory aimed for singular stochastic partial differential equations (SPDEs), which was firstly introduced by Gubinelli, Imkeller and Perkowski
\cite{GIP}. Such
paracontrolled theory has also been efficiently applied to different types of
SDEs
with singular coefficients beyond the Young regime, see e.g. Cannizzaro and Chouk \cite{CC}, Kremp and Perkowski \cite{KP}, Zhang, Zhu and Zhu \cite{ZZZ}.
Roughly speaking, to solve \eqref{e0} through the paracontrolled theory, some extra techniques are required to tackle the growth property of $\{\dot{W}(x)\}_{x\in \R}$.

With aid of the representation \eqref{brox}, Schumacher and Brox proved, independently in  \cite{Sch}  and \cite{Brox}, that for large $t$
the average value of
$X(t)$ under the annealed  probability $\Pp^0$
is much smaller than $t^{1/2}$, the magnitude order of a standard Brownian motion in non-random environment.
In fact, the average value of $X(t)$
in the annealed setting
is of order $(\log t)^2$, which is surprisingly slow.
Therefore, the Brox diffusion $\{X(t)\}_{t\ge 0}$ describes a Brownian motion moving in a random medium, and it will  possess anomalous behaviors. After the work of \cite{Brox,Sch} there have been a number of papers devoted to the study of the
Brox diffusion.
Tanaka \cite{Tan1, Tan2} studied different localization behaviors for the Brox diffusion.
Hu and Shi \cite{HS} established law of the iterated logarithm for the Brox diffusion, as well as
the moderate
deviation in \cite{HS3}. The scaling limit from Sinai's random walk
to the Brox diffusion was proved by Seignourel \cite{Se}. Via the coupling method and the rough path theory,
recently a rate for such convergence has been given by Geng, Gradinaru and Tindel \cite{GGT}.
We also refer the reader to
\cite{Che, HS2, HSM, M, S} about various properties of the Brox diffusion.

\subsection{Main results}
The purpose of this paper is to establish heat kernel estimates for the Brox diffusion
$\{X(t)\}_{t\ge 0}$,
including
both the small time and large time.
It is easy to see that the Brox diffusion $\{X(t)\}_{t\ge 0}$ is a $\mu_X$-symmetric diffusion on $\R$, where $\mu_X(dx)=\exp (-W(x)))\,dx$.
Since $\{W(x)\}_{x\in \R}$ is locally bounded, it then follows from the standard theory of Dirichlet forms (see e.g. \cite{FOT11}) that there exists
a heat kernel (i.e., a transition density function) of the Brox diffusion $\{X(t)\}_{t\ge 0}$ with respect to the reference measure $\mu_X$, which is denoted by
$p^X(t,x,y)$ or $p^X(t,x,y,\omega)$ in this paper.

First, we have the following quenched estimates of the heat kernel $p^X(t,x,y,\omega)$.
\begin{theorem}\label{t1-1} For any $\alpha\in (0,1/2)$, there exist positive
constants $C_i$, $i=1,\cdots,4$, such that the following estimates hold.
\begin{itemize}
\item[{\rm(i)}] There are positive random variables $C_5(\omega)$ and $C_6(\omega)$ such that for every
$x,y\in \R$, $t\in (0,1]$ and almost all $\omega \in \Omega$,
\begin{align*}
p^X(t,x,y,\omega)\ge & C_1t^{-1/2}\exp\left(-\frac{C_2|x-y|^2}{t}\right)e^{W(y,\om)} \exp\left(-C_5(\omega) t^2[\log(2+|x|+|y|)]^{{2}/{\alpha}} \right)
\end{align*} and
\begin{align*}
p^X(t,x,y,\omega) \le & C_3t^{-1/2}\exp\left(-\frac{C_4|x-y|^2}{t}\right)e^{W(y,\om)}\exp\left(C_6(\omega) t^3[\log(2+|x|+|y|)]^{ {3}/{\alpha}} \right).
\end{align*}
\item[{\rm (ii)}] There are a positive random variable $R_0(\om)$ and positive
constants $C_7$, $C_8$ such that for every
$x,y\in \R$ with $\max\{|x|,|y|\}>R_0(\om)$, $t\in (0,1]$ and almost all $\omega \in \Omega$,
\begin{align*}
p^X(t,x,y,\omega)\ge & C_1t^{-1/2}\exp\left(-\frac{C_2|x-y|^2}{t}\right)e^{W(y,\om)} \exp\left(-C_7 t^2[\log(2+|x|+|y|)]^{{2}/{\alpha}} \right)
\end{align*} and
\begin{align*}
p^X(t,x,y,\omega)\le C_3t^{-1/2}\exp\left(-\frac{C_4|x-y|^2}{t}\right)e^{W(y,\om)}\exp\left(C_8t^3[\log(2+|x|+|y|)]^{{3}/{\alpha}} \right).
\end{align*}
\end{itemize}
\end{theorem}

As a consequence of Theorem \ref{t1-1}, we have the statement as follows.

\begin{corollary}\label{c1-1} The following quenched estimates of $p^X(t,0,x,\omega)$ hold.
\begin{itemize}
\item[{\rm(i)}] There exist   positive random variables $C_i(\omega)$, $9\le i\le 12$, such that
 for every
$x\in \R$, $t\in (0,1]$ and almost all $\om \in \Omega$,
$$
C_9(\omega)t^{-1/2}\exp\left(-\frac{C_{10}(\omega)|x|^2}{t}\right)\le
p^X(t,x,0,\omega)
\le C_{11}(\omega)t^{-1/2}\exp\left(-\frac{C_{12}(\omega)|x|^2}{t}\right).
$$
\item[{\rm(ii)}] There are a positive random variable $R_0(\om)$ and
positive constants $C_{i}$, $13\le i\le 16$, such that for every
$x\in \R$ with $|x|>R_0(\om)$, $t\in (0,1]$ and almost all $\om \in \Omega$,
$$
C_{13}t^{-1/2}\exp\left(-\frac{C_{14}|x|^2}{t}\right)\le
p^X(t,x,0,\omega)
\le C_{15}t^{-1/2}\exp\left(-\frac{C_{16}|x|^2}{t}\right).
$$\end{itemize}
\end{corollary}

\ \
Second, let $p(t,x,y)$ be the heat kernel of the Brox diffusion $X$ with respect to the Lebesgue measure; that is, for any $x,y \in \R$ and $t>0$,
\begin{align}\label{heat}p(t,x,y)=p^X(t,x,y)e^{-W(y)}.\end{align}
The following result is devoted to annealed estimates of $p(t,x,x)$ for large time.

\begin{theorem}\label{thm22} For any $\alpha\in (0,1/2)$, there exists
constants  $T_0,C_0\ge1$ such that for every $x\in \R$ and $t\ge T_0$,
$$\frac{1}{C_0(\log t)^2(\log\log t)^{11}}\le \Ee\left[p(t,x,x)\right]\leq \frac{C_0(\log\log t)^{4+1/(2\alpha)}}{(\log t)^2}.$$\end{theorem}

\begin{remark}\rm We make some comments on the results above.
\begin{itemize}
\item [{\rm (i)}] Recently, heat kernel estimates for diffusion processes in ergodic random environments have been investigated in
\cite{ADS1,ADS2,B,BCh,BCK,BBHK,Bo,CHK1,CHK}. In particular, these estimates enjoy quite different forms in comparison with the main results of our paper for the Brox diffusion. The reasons are as follows. In the present setting, several
regular conditions,
especially the volume doubling conditions with respect to the reference measure $\mu_X$,
do not hold neither for the small scale nor the large scale;
see Remark \ref{e:addvd} below for details.
Due to the oscillation
property of $\{W(x)\}_{x\in \R}$ and its H\"older coefficients, the techniques in terms of good ball conditions (\cite{B,BCh}),
the integrated version of Davies' method (\cite{ADS1,ADS2}), and some uniform control of escape probabilities
(\cite{CHK1,CHK}) could not be applied to the Brox diffusion. We shall develop
new methods to tackle these difficulties for heat kernel estimates of the Brox diffusion.

\item[{\rm(ii)}]  The generator of the Brox diffusion $X$ is symmetric with respect to $\mu_X(dx) =\exp(-W(x))\,dx$. So, the random term $e^{W(y,\om)}$ naturally appears in both-sided quenched estimates for the heat kernel $p^X(t,x,y,\om)$ in Theorem \ref{t1-1}. On the other hand, in Theorem \ref{t1-1}(i) the random perturbation terms in upper and
    lower bounds of the quenched estimates, such as $\exp\big(-C_5(\omega)t^2[\log(2+|x|+|y|)]^{{2}/{\alpha}}\big)$ and
    $\exp\big(C_6(\omega) t^3[\log(2+|x|+|y|)]^{ {3}/{\alpha}} \big)$, arise from the growth property for the H\"older coefficients $\Xi(x,r,\om)$ and $\Upsilon(r,c_0,\om)$,  defined
    by \eqref{e3-1a} and \eqref{e3-3} respectively. Moreover, Corollary \ref{c1-1}(ii) reveals that
    in the regimes of finite time and large distance, we can obtain the Gaussian type two-sided estimates with non-random coefficients
    for the heat kernel $p(t,0,x)$ of the Brox diffusion.

\item[{\rm(iii)}]  The leading term for annealed estimates of $p(t,x,x)$ for large time $t$ is of order
$(\log t)^{-2}$, which in some sense is consistent with Schumacher and Brox's annealed results for the growth of the Brox diffusion $X$ for large time $t$ (which is with the order $(\log t)^2$).
As indicated in the proof of Theorem \ref{thm22},
such term is determined by the dominated event that $W$ firstly reaches a relatively large positive value with the probability of the order $\log t$ for large $t$, in comparison with the corresponding negative value (which is the valley of $\{W(x)\}_{x\in \R}$). In this sense, we know that anomalous behaviors for the heat kernel of the Brox diffusion $X$ mainly come from the large oscillation
of the positive value for $W$,
which are completely different from the trap models studied by \cite{BC,BeC,BBHK,Bo}. Indeed, for trap models or other random walks in random media, the annealed (on-diagonal) heat kernel estimates usually have tight asymptotic behaviors, and the fluctuation results like Theorem \ref{thm22} only occur for quenched large time heat kernel estimates, see \cite[Theorems 1.2 and 1.4]{ACK}. The later is associated with the corresponding fluctuations of the volume, which is broken down for Brox's diffusion since the volume doubling does not hold in large scale. See the survey paper \cite{ACK2} and references therein on the related topic.

\item[{\rm(iv)}]  A very precise image of the almost sure asymptotic behaviors of
Brox's diffusion has been established in \cite[Theorems 1.6, 1.7 and 1.8]{HS}. In particular, the limsup or the liminf of the sample path asymptotics is of the order $(\log t)^2$ with positive or negative power of $(\log\log\log t)$ as a lower
order
perturbation. These statements further corresponds to annealed heat kernel estimates stated in Theorem \ref{thm22}. We shall mention that heat kernel estimates
seem more complicated
than the probability estimates involved in asymptotic behaviors, since they are concerned on the estimates for transition density functions.
\end{itemize}\end{remark}

\subsection{Approach} We briefly illustrate the approaches of our main results.
According to Brox's construction above, formally $S(x)$ is the scale function of the Brox diffusion $\{X(t)\}_{t\ge 0}$, and $T(t)$ is a positive continuous additive functional of the Brownian motion $\{B(t)\}_{t\ge 0}$. Thus, $\{Y(t)\}_{t\ge 0}:=\{B(T^{-1}(t))\}_{t\ge0}$  is a time-change of $\{B(t)\}_{t\ge 0}$, which is a $\mu_Y$-symmetric strong Markov process on $\R$ with $\mu_Y(dx)=\exp (-2W(S^{-1}(x)))\,dx$; see  \cite[Theorem 5.2.2]{CF}.
On the other hand, we know that a time-change of Brownian motion  does not change its transience
and recurrence (see \cite[Theorem 5.2.5]{CF}). It follows from \cite[Corollaries 3.3.6 and 5.2.12]{CF} that
the Dirichlet from $(\mathcal{E}_Y,\cF_Y)$ on $L^2(\mu_Y)$ associated with the process $Y$ is given by
$$\mathcal{E}_Y(f,f)=\frac12\int_\mathbb{R}f'(x)^2\,dx,\quad \cF_Y=\{f\in L^2(\mu_Y):\cE_Y(f,f)<\infty\}.$$

As mentioned in the beginning of the previous subsection, the Brox diffusion $X$ is a $\mu_X$-symmetric Markov process on $\R$ with $\mu_X(dx)=\exp (-W(x)))\,dx$.
Denote by $p^X(t,x,y)$ (resp.\ $p^Y(t,x,y)$) the heat kernel of the Brox process $X$ with respect to $\mu_X$ (resp. the time-change process $Y$ with respect to $\mu_Y$). Then, by the fact $X(t)=S^{-1}(Y(t))$, for any $f\in C_b(\mathbb{R})$, $t>0$ and $x\in \R$,
\begin{align*}
\int_\mathbb{R}p^X(t,x,y)f(y)\,\mu_X(dy)&=\bE(f(X(t))|X(0)=x)=\bE (f(S^{-1}(Y(t)))|Y(0)=S(x))\\
&=\int_\mathbb{R}p^Y(t,S(x),y)f(S^{-1}(y))\,\mu_Y(dy)\\
&=\int_\mathbb{R}p^Y(t,S(x),y)f(S^{-1}(y))\exp (-2W(S^{-1}(y)))\,dy\\
&=\int_\mathbb{R}p^Y(t,S(x),S(y))f(y)e^{-W(y)}\,dy\\
&=\int_\mathbb{R}p^Y(t,S(x),S(y))f(y)\,\mu_X(dy),
\end{align*}
which implies that for any $t>0$ and $x,y\in \R$,
\begin{align}\label{eqxy}
p^X(t,x,y)=p^Y(t,S(x),S(y)).
\end{align}
Thus, in order to obtain the estimates for $p^X(t,x,y)$ we turn to
those
for $p^Y(t,x,y)$, which is the heat kernel corresponding to a time-change of recurrent Brownian motion $\{Y(t)\}_{t\ge 0}$. Furthermore, we will make use of the approach through the theory of resistance forms for strongly recurrent Markov processes (see \cite{BCK, T04}) to establish the estimates of $p^Y(t,x,y)$. The crucial ingredient in our proof is to obtain suitable estimates of $V(S(x),R)$. Here, $V(x,R)$ is the volume of the ball induced by the reference measure $\mu_Y(dx)$. We emphasize that, in the present setting the so-called volume doubling conditions do not hold. Furthermore, for small time quenched
estimates, we derive the escape probabilities by making full use of the growth property of H\"older's coefficients for the Brownian sample path. While for large time annealed
estimates, we will introduce a proper decomposition of the probability space (i.e., environments) related to the valley of the Brownian motion $W$. Herein, we also apply a few known estimates for the functional of Brownian motion.

\ \

The rest of the paper is arranged as follows. Section \ref{section2} is devoted to some preliminary estimates for the time-change process $\{Y(t)\}_{t\ge 0}$. With the aid of these estimates, we present the proofs of Theorems \ref{t1-1} and \ref{thm22} is Sections \ref{section3} and \ref{section4}, respectively. Throughout the proofs, we will omit the variable $\om$ from time to time if no confusion is caused, and
all the constants $C_i$ or $c_i$ are non-random without particularly clarification.

\section{Preliminary estimates for the time-change process $\{Y(t)\}_{t\ge 0}$}\label{section2}

With the aid of \eqref{eqxy}, in order to establish heat kernel estimates for the Brox diffusion $\{X(t)\}_{t\ge 0}$ we shall consider bounds for the heat kernel $p^Y(t,x,y)$ of the time-change process $\{Y(t)\}_{t\ge 0}$. Recall that $\mu_Y(dx)=\exp (-2W(S^{-1}(x)))\,dx$ is the reference measure for the process $\{Y(t)\}_{t\ge 0}$, where $S(x)=\displaystyle \int_0^x e^{W(z)}\,dz$.
In the following, for every $x\in \R$ and $R>0$, define
$$B(x,R):=(x-R, x+R),\quad V(x,R):=\mu_Y(B(x,R)).$$ Similarly,
define
$
B_+(x,R):=[x, x+R),$ $ V_+(x,R):=\mu_Y(B_+(x,R))$, $B_-(x,R):=(x-R,x]$ and $V_-(x,R):=\mu_Y(B_-(x,R)).
$

We begin with the following elementary properties.
\begin{lemma}\label{l2-1}
For every $x\in \R$, it holds almost surely that
\begin{equation}\label{l2-1-1}
\lim_{R \to \infty}V_+(x,R)= \lim_{R \to \infty}V_-(x,R)=\lim_{R \to \infty}V(x,R)=\infty
\end{equation} and
\begin{equation}\label{l2-1-2}
\lim_{t \to 0}p^Y(t,x,x)=\infty,\quad \lim_{t \to \infty}p^Y(t,x,x)=0.
\end{equation}
\end{lemma}

We will firstly introduce some some notations.
For any $a<b$, define
$$\xi_0(a,b)=\sup_{a\le s\le t\le b} |W(s,\om)-W(t,\om)|,\quad \om \in \Omega.$$ For simplicity, write $\xi_0(a)=\xi_0(0,a)$ for all $a>0.$ Then, $\xi_0(a,b)$ has the same distribution as $\xi_0(0,b-a)$. According to \cite[(2.4)]{Ta72}, there is a constant $c_0>0$ such that for all $\lambda, a>0$,
\begin{equation}\label{e:es}\Pp(\xi_0(a)\ge \lambda a^{1/2})\le c_0 \lambda^{-1}\exp\left(-\lambda^2/2\right).\end{equation}  Define
\begin{equation}\label{brownian holder1}
\xi(x,r;\omega):=\sup_{s,t\in [x-r,x+r]}{\left|W(s,\om)-W(t,\om)\right|},\quad  \omega\in \Omega,\ x\in \R,\, r\ge0.
\end{equation}
Throughout this paper, we will fix $\alpha\in (0,1/2)$, and we set
\begin{equation}\label{e3-1a}
\Xi(x,r;\omega):=\sup_{s,t\in [x-r,x+r]}\frac{\left|W(s)-W(t)\right|}{|t-s|^{\alpha}},\quad  \omega\in \Omega,\ x\in \R,\ r>0.
\end{equation}
It is clear that
\begin{equation}\label{e3-2}
\xi(x,s;\omega)\le  \Xi(x,r;\omega)(2s)^\alpha,\ \ \omega\in \Omega,\ x\in \R,\ 0<s\le r.
\end{equation}

\begin{proof} [Proof of Lemma $\ref{l2-1}$]
For simplicity we only prove the assertion when $x=0$, and the conclusion for
all $x\in \R$ can be proved by the same way.

(i) Let
$$
\xi_{0,n}:=\xi_0(n,n+1)=\sup_{n\le s,t\le n+1}{|W(s)-W(t)|},\quad n\ge 0.
$$
By the property of the Brownian motion $W$ and \eqref{e:es},  $\{\xi_{0,n}\}_{n\ge 1}$ is a sequence of
i.i.d. random variables so that $$
\sum_{n=0}^\infty\Pp\left(|\xi_{0,n}|>\sqrt{4\log (2+n)}\right)
\le c_0\sum_{n=0}^\infty \frac{1}{(2+n)^2}<\infty.
$$
Combining this with Borel-Cantelli' lemma yields that there is a random integer $n_0:=n_0(\om)>0$ such that almost surely
\begin{equation}\label{l2-1-3}
|\xi_{0,n}|\le \sqrt{4\log (2+n)},\quad n>n_0.
\end{equation} On the other hand, according to the law of iterated logarithm for the Browian motion $W$, it holds that almost surely
$$
\varliminf_{s \to \infty}\frac{W(s)}{\sqrt{2s\log\log s}}=-1,\quad
\varlimsup_{s \to \infty}\frac{W(s)}{\sqrt{2s\log\log s}}=1.
$$
This together with  \eqref{l2-1-3} gives us that there are two  random  sequences $\{s_n^+\}_{n\ge 1}:=\{s_n^+(\om)\}_{n\ge 1}$ and  $\{s_n^-\}_{n\ge 1}:=\{s_n^-(\om)\}_{n\ge 1}$ such that
\begin{equation}\label{l2-1-4}
\begin{split}
& s_{n+1}^+-s_n^+\ge 4,\quad \inf_{s\in [s_n^+,s_n^++1]}W(s)
/\sqrt{s}\ge 1,\\
& s_{n+1}^--s_n^-\ge 4,\quad \sup_{s\in [s_n^-,s_n^-+1]}W(s)
/\sqrt{s}\le -1.
\end{split}
\end{equation}

Recall that
$$
\mu_Y(dy)=\exp (-2W(S^{-1}(y)))\,dy,\quad  S(y)=\int_0^y e^{W(z)}\,dz.
$$
By \eqref{l2-1-4}, we know immediately that $S(\cdot)$ is strictly increasing so that almost surely
\begin{align}\label{l2-1-5}
\lim_{y \to \infty} S(y)=\infty,\quad \lim_{y\to\infty}S^{-1}(y)=\infty.
\end{align}
By the change of variable $z=S^{-1}(y)$, it holds that
$$
V_+(0,R)=\int_0^R \exp (-2W(S^{-1}(y)))\,dy=\int_0^{S^{-1}(R)}\exp(-W(z))\,dz.
$$
Hence, it follows from \eqref{l2-1-4} and \eqref{l2-1-5} that almost surely
$$
\lim_{R \to +\infty}V_+(0,R) =\lim_{y \to  \infty}\int_0^{y}\exp(-W(z))\,dz\ge \sum_{n=1}^\infty \int_{s_n^-}^{s_n^-+1}\exp(-W(z))\,dz\ge \sum_{n=1}^\infty \inf_{s\in [s_n^-,s_n^-+1]}e^{\sqrt{s}}=\infty.
$$

By the same argument, we  also can prove other conclusions in \eqref{l2-1-1}, and so we omit
the details here.

(ii) For any bounded subset $U \subset \R$, let $p^Y_U(\cdot,\cdot,\cdot):\R_+\times U \times U \to \R_+$ be the Dirichlet heat kernel
associated with the process $Y$ killed on exiting from $U$. In particular, this subprocess  corresponds to the Dirichlet form $\left(\mathcal{E}_Y^U, \cF^U\right)$ on $L^2(U;\mu_Y)$ as follows:
$$\mathcal{E}_Y^U(f,f)=\frac12\int_U |f'(x)|^2\,dx,\quad \cF^U=\overline{C_c^\infty(U)}^{(\mathcal{E}_Y^U(\cdot,\cdot)+\|\cdot\|_{L^2(U;\mu_Y)}^2)^{1/2}}.$$
This is, $\cF^U$ is the closed extension
of $C_c^\infty(U)$ under the norm $(\mathcal{E}_Y^U(\cdot,\cdot)+\|\cdot\|_{L^2(U;\mu_Y)}^2)^{1/2}$.

Below, we choose $U=B(0,1)$. According to \eqref{l2-1-5}, we know that almost surely the density function (with respect to the Lebesgue measure) of $\mu_Y$
is locally bounded in $\R$. Thus, there exist  positive random   variables $ c_1(\om)$ and $ c_2(\om)$ such that
$$
c_1(\om) |A|\le \mu_Y(A)\le c_2(\om) |A|,\quad A \subset B(0,1),
$$
where $|A|$ denotes the Lebesgue measure of $A$.
In particular,
$(\mathcal{E}_Y^{B(0,1)}(\cdot,\cdot)+\|\cdot\|_{L^2(B(0,1);\mu_Y)}^2)^{1/2}$ is comparable to the norm associated with Brownian motion killed on exiting from $B(0,1)$. Applying
the standard result (see \cite[Chapter 5]{S}), $$
p^Y_{B(0,1)}(t,0,0)\ge c_3(\om) t^{-1/2},\quad  t\in (0,1/2].
$$
Combining this with the fact $p^Y(t,0,0)\ge p^Y_{B(0,1)}(t,0,0)$, we establish
the first assertion in \eqref{l2-1-2}.

(iii)
Since $\displaystyle\int_{B_+(0,r)}p^Y(t,0,y)\,\mu_Y(dy)\leq 1$ for any $t,r>0$, there exists $y_0:=y_0(t,r,\om)\in B_+(0,r)$ such that
$$p^Y(t,0,y_0)\leq \frac1{V_+(0,r)}.$$
Define
$$\psi(t)=\int_\mathbb{R}p^Y(t,0,y)^2\,\mu_Y(dy),\quad  t>0.$$
Then, by the symmetry of $p^Y(t,x,y)$ with respect to $(x,y)$,
$$\psi(t)=\int_\mathbb{R}p^Y(t,0,y)p^Y(t,y,0)\,\mu_Y(dy)=p^Y(2t,0,0).$$
Note that, for any $f\in \cF$ and $x,y\in \R$,
\begin{equation}\label{l2-1-5a}
|f(x)-f(y)|\leq \bigg|\int_x^y f'(z)\,dz\bigg|\leq  \sqrt{2}\mathcal{E}_Y(f,f)^{1/2}|x-y|^{1/2}.
\end{equation}
We then get that for every $t>0$ and $r>0$,
\begin{align*}\frac12p^Y(t,0,0)^2\leq& p^Y(t,0,y_0)^2+|p^Y(t,0,0)-p^Y(t,0,y_0)|^2
\\\leq&\frac1{V_+(0,r)^2}+2\mathcal{E}_Y(p^Y(t,0,\cdot),p^Y(t,0,\cdot))|y_0|
\\\leq&\frac1{V_+(0,r)^2}+2\mathcal{E}_Y(p^Y(t,0,\cdot),p^Y(t,0,\cdot))r,
\end{align*}
which implies that for $t>0$
$$\psi'(t)=-2\mathcal{E}_Y(p^Y(t,0,\cdot),p^Y(t,0,\cdot))\leq \frac1 {2r}\bigg(\frac2{V_+(0,r)^2}-\psi(t/2)^2\bigg).$$
This along with  $\psi(t/2)\ge \psi(t)$, due to the fact that $\psi'(t)=-2\mathcal{E}_Y(p^Y(t,0,\cdot),p^Y(t,0,\cdot))\le 0$,
further yields that for every $t>0$ and $r>0$,
\begin{align}\label{l2-1-6}
\psi'(t)\leq \frac1{2r}\bigg(\frac2{V_+(0,r)^2}-\psi(t)^2\bigg).
\end{align}
Set $F(r):=\frac{1}{V_+(0,r)^2}$ and $F^{-1}(r):=\inf\{t>0: F(t)\le r\}$ for each $r>0$. By \eqref{l2-1-1},
$$\lim_{r \to \infty}F(r)=0,\quad \lim_{r\to 0}F^{-1}(0)=+\infty.$$
Next, we prove that it holds almost surely
\begin{equation}\label{l2-1-7}
\sup_{t\ge 1}F^{-1}(4^{-1}\psi(t)^2)=\infty.
\end{equation}
Indeed, if \eqref{l2-1-7} holds, then we can find $K>0$ large enough and  $t_0\ge 1$  so that $F^{-1}(4^{-1}\psi(t_0)^2)\ge K$. This, along
with the non-increasing properties of $\psi(\cdot)$ and $F(\cdot)$, implies that
$$
\inf_{t\ge 1}\psi(t)^2\le 4F(K).
$$
Then, letting $K \to \infty$, we obtain the second assertion in \eqref{l2-1-2}, thanks to the decreasing property of $\psi(t)$ again.

We will verify \eqref{l2-1-7} by contradiction. Now suppose that \eqref{l2-1-7} is not true. Then,  there exists $K_0>0$ such that
$$
F^{-1}(4^{-1}\psi(t)^2)< K_0,\quad t\ge 1,
$$
which implies that
\begin{align}\label{l2-1-8}
\psi(t)^2> 4F(K_0),\quad  t\ge 1.
\end{align}
Therefore, taking $r=K_0$ in \eqref{l2-1-6}, we obtain
$$
\psi'(t)\le -\frac{\psi(t)^2}{4K_0},\quad t\ge 1,
$$
from which we can deduce that
$$
\psi(t)\le \frac{4K_0}{t-1},\quad t\ge 2.
$$
Now, taking $t\to \infty$, we know that the estimate above contradicts with \eqref{l2-1-8}. Thus \eqref{l2-1-7} holds, and so we finish the proof.
\end{proof}

As mentioned above, the process $\{Y(t)\}_{t\ge 0}$ can be regarded as a time-change of the standard Brownian motion
$\{B(t)\}_{t\ge 0}$. Below we will make full use of the approaches in \cite{BCK,T04} to obtain
heat kernel estimates of the process $Y$.

\bl\label{upper} {\bf (On-diagonal upper bounds)} For any $x\in \R$ and $R>0$, almost surely
\begin{equation}\label{l2-2-1}
p^Y\left(4V(x,R)R,x,x\right)\leq \frac2{V(x,R)}
\end{equation} and
\begin{equation}\label{l2-2-1a}
p^Y\left(4V_+(x,R)R,x,x\right)\leq \frac2{V_+(x,R)},\quad p^Y\left(4V_-(x,R)R,x,x\right)\leq \frac2{V_-(x,R)}.
\end{equation}
\el
\begin{proof}
For simplicity, we only prove the first assertion in \eqref{l2-2-1a} with $x=0$. The other cases (including the first assertion in \eqref{l2-2-1a} for general $x\in \R$) can be proved by exactly the same way.

Define
$$\psi(t):=\int_\mathbb{R}p^Y(t,0,y)^2\,\mu_Y(dy)=p^Y(2t,0,0),\quad t>0.$$
According to the proof of Lemma \ref{l2-1}, we know that \eqref{l2-1-6} holds, i.e.,
$$
\psi'(t)\leq \frac1{2r}\bigg(\frac2{V_+(0,r)^2}-\psi(t)^2\bigg),\quad t,r>0.
$$

Set $\varphi(t)=2/{\psi(t)}$, so $\varphi$ is increasing as $\psi$ is decreasing
(we have mentioned that $\psi'\le 0$ in the proof of Lemma \ref{l2-1}). Moreover,  it holds that
\begin{align}\label{l2-2-2}
\varphi'(t)=-\frac12\varphi^2(t)\psi'(t)\geq \frac1{2r}\big(2-\varphi(t)^2V_+(0,r)^{-2}\big).
\end{align}
For every fixed $t>0$, by \eqref{l2-1-1}, \eqref{l2-1-2} and the fact that $r\mapsto V_+(0,r)$
is continuous and non-decreasing, we can find $r(t)>0$ such that $\varphi(t)=V_+\left(0,r(t)\right)$.
Hence, taking $r=r(t)$ in \eqref{l2-2-2} yields that
$$
\varphi'(t)\geq \frac1{2r(t)},\quad t>0.
$$
Since $t \mapsto \varphi(t)$ is increasing,  $t\mapsto r(t)$ is increasing. So, taking
integration on the both side of the inequality above, we arrive at
$$V_+(0,r(t))r(t)=\varphi(t)r(t)\ge \int_0^t\varphi'(s)r(s)\,ds\geq t/2,\quad t>0,$$
which implies that
$$\varphi\left(2V_+(0,r(t))r(t)\right)\geq \varphi(t)=V_+(0,r(t)),\quad t>0.$$
Thus, according to the definition of $\varphi$, we obtain
\begin{align}\label{l2-2-3}
p^Y\left(4V_+(0,r(t))r(t),0,0\right)\leq \frac2{V_+(0,r(t))},\quad  t>0.
\end{align}

Furthermore, according to  \eqref{l2-1-1} and \eqref{l2-1-2} again, we have immediately that $\lim_{t \to \infty}r(t)=\infty$. Then, for any given $R>0$ we can find $t:=t(R)>0$ such that $r(t)=R$.
This along with \eqref{l2-2-3} proves the desired assertion.
\end{proof}

To get on-diagonal lower bounds, we need estimates for the mean of the exit time. For any $D\subset \R$ and $r>0$,
we define $\tau^Y_{D}:=\inf\{t\ge0: Y(t)\notin D\}$.
\begin{lemma}\label{l2-3}
There exist positive  constants $C_1$ and $C_2$ such that for every $z\in \R$ and $r>0$, almost surely
\begin{equation}\label{stopping}
\bE^x[\tau^Y_{B(z,r)}]\le C_1rV(z,r),\quad x\in B(z,r)
\end{equation} and
\begin{equation}\label{stopping1}
\begin{split}
\bE^x[\tau^Y_{B(z,r)}]\geq C_2rV(z,r/2),\quad x\in B(z,r/2).
\end{split}
\end{equation}
\end{lemma}
\begin{proof}
Let $G_D^Y(x,y)$ and $G_D^B(x,y)$ be the Green function for the process $Y$ and the Brownian motion
$B$
on the open set $D$, respectively.
As $\{Y(t)\}_{t\ge 0}$ is a time-change of the Brownian motion $\{B(t)\}_{t\ge 0}$ and Green function is invariant under time-change (see \cite[Exercise 4.2.2, Lemma 5.1.3 and the first paragraph in p.\ 362]{FOT11}), we have
$$
\bE^x \left[\tau^Y_D\right]=\int_D G_D^Y(x,y)\,\mu_Y(dy)=\int_D G_D^B(x,y)\,\mu_Y(dy),\quad x\in D.
$$
Moreover, it is well known that (see \cite[p.\ 45]{CZ})
there are  constants $c_1, c_2>0$ so that for all $z\in \R$ and $r>0$,
$$
G^Y_{B(z,r)}(x,y)=G^B_{B(z,r)}(x,y)\leq c_1r, \quad x,y\in B(z,r),
$$
and
$$
G^Y_{B(z,r)}(x,y)=G^B_{B(z,r)}(x,y)\geq c_2r, \quad x,y\in B(z,r/2).
$$
Then, putting all the estimates above together, we find that for any $z\in \mathbb{R}$ and $r>0$,
$$
\bE^x[\tau^Y_{B(z,r)}]=\int_{B(z,r)}G^Y_{B(z,r)}(x,y)\,\mu_Y(dy)\leq c_1r\mu_Y(B(z,r)),\quad x\in B(z,r)$$
and
$$
\bE^x[\tau^Y_{B(z,r)}]\geq\int_{B(z,r/2)}G^Y_{B(z,r)}(x,y)\,\mu_Y(dy)\geq c_2r\mu_Y(B(z,r/2)),\quad x\in B(z,r/2).
$$
This finishes the proof.
\end{proof}

\bl\label{lower}{\bf (On-diagonal lower bounds)} For $x\in \R$ and $R>0$,  almost surely
\begin{align}\label{l2-4-0}
p^Y(2C_2RV(x,R),x,x)\geq \frac{C_2^2V(x,R)^2}{4C_1^2V(x,2R)^3},
\end{align}
where $C_1$ and $C_2$ are  given in
\eqref{stopping} and \eqref{stopping1} respectively.
\el
\begin{proof}By \eqref{stopping} and \eqref{stopping1}, for any $x\in \R$ and $R, t>0$,
\begin{align*}
C_2RV(x,R/2)\leq& \bE^x[\tau^Y_{B(x,R)}]\leq t+\bE^x[\I_{\{\tau^Y_{B(x,R)}>t\}}\bE^{Y(t)}[\tau^Y_{B(x,R)}]]
\\\leq&t+C_1RV(x,R)\bP^x(\tau^Y_{B(x,R)}>t),
\end{align*}
which implies that when $t\leq C_2RV(x,R/2)/2$
$$
\bP^x(\tau^Y_{B(x,R)}>t)\geq \frac{C_2RV(x,R/2)-t}{C_1RV(x,R)}\geq \frac{C_2V(x,R/2)}{2C_1V(x,R)}.
$$
Denote by $p^Y_{D}(t,x,y)$ the Dirichlet heat kernel associated with the process $\{Y(t)\}_{t\ge 0}$.
Since, for any $t>0$,
\begin{align*}
 \big(\bP^x(\tau^Y_{B(x,R)}>t) \big)^2=&\left(\int_{B(x,R)}p^Y_{B(x,R)}(t,x,z)\,\mu_Y(dz)\right)^2
\\\leq&\left( \int_{B(x,R)}(p^Y_{B(x,R)}(t,x,z))^2\,\mu_Y(dz)\right) \mu_Y(B(x,R))
\\=&p^Y(2t,x,x)V(x,R),
\end{align*}
we find that when $t\leq C_2RV(x,R/2)/2$,
$$p^Y(2t,x,x)\geq \frac{ (\bP^x(\tau^Y_{B(x,R)}>t) )^2}{V(x,R)}\geq \frac{C_2^2V(x,R/2)^2}{4C_1^2V(x,R)^3}.$$
This proves the desired assertion by taking $t=C_2RV(x,R/2)/2$ in the inequality above.
\end{proof}

According to \eqref{eqxy} and Lemmas \ref{upper} and \ref{lower} above, suitable estimates of $V(S(x),R)$ are key ingredients for estimates of $p^X(t,x,x)$. For our purpose,
we will establish the estimates for $V(S(x),R)$ in Lemma \ref{l3-1} below.
For every $x\in \R$ and $R>0$, set
\begin{equation}\label{e3-1}
\begin{split}
&\delta_+(x,R;\omega):=S^{-1}\left(S(x)+R\right)-x=\inf\left\{y>0:\int_x^{x+y} e^{W(z,\omega)}dz=R\right\},\\
&\delta_-(x,R;\omega):=x-S^{-1}\left(S(x)-R\right)=\inf\left\{y>0:\int_{x-y}^x e^{W(z,\omega)}dz=R\right\}.
\end{split}
\end{equation}

\begin{lemma}\label{l3-1}
For $x\in \R$ and $R>0$, it holds almost surely that
\begin{equation}\label{volume1}
\begin{split}
&2Re^{-2W(x)}e^{-2\xi(x,(2RV(S(x),R))^{1/2})}\leq V(S(x),R)
\leq  2Re^{-2W(x)}e^{2\xi(x,(2RV(S(x),R))^{1/2})},
\end{split}
\end{equation}
where $\xi(x,r)$ is defined by \eqref{brownian holder1} above.
\end{lemma}

\begin{proof}
Recall that for any $x\in \R$ and $R>0$,
$$V(S(x),R)=\mu_Y(B(S(x),R))= \int_{S(x)-R}^{S(x)+R}\exp (-2W(S^{-1}(y)))\,dy=\int_{S^{-1}(S(x)-R)}^{S^{-1}(S(x)+R)}\exp (-W(y))\,dy.
$$
In the following, we write $\delta_+$, $\delta_-$ and
$\xi(r)$
for
$\delta_+(x,R)$, $\delta_-(x,R)$ and
$\xi(x,r)$
respectively.
According to the definition of $\delta_+$ given by \eqref{e3-1},  we have
\begin{align}\label{eqr}
\int_x^{x+\delta_+}e^{W(z)}\,dz=R.
\end{align}
By \eqref{brownian holder1}, we know that
$$\max\left\{\int_x^{x+\delta_+}e^{W(z)-W(x)}\,dz, \int_x^{x+\delta_+}e^{W(x)-W(z)}\,dz\right\} \le
\delta_+ e^{\xi(\delta_+)}$$
and
$$\min\left\{\int_x^{x+\delta_+}e^{W(z)-W(x)}\,dz, \int_x^{x+\delta_+}e^{W(x)-W(z)}\,dz\right\} \ge
\delta_+ e^{-\xi(\delta_+)}.
$$
Hence, by \eqref{eqr},
\begin{align*}
\delta_+e^{-\xi(\delta_+)}\leq & Re^{-W(x)}=\int_x^{x+\delta_+}e^{W(z)-W(x)}\,dz
\leq \delta_+e^{\xi(\delta_+)}.
\end{align*}
This in turn yields that
\begin{align*}\int_x^{S^{-1}(S(x)+R)}e^{-W(y)}\,dy&=\int_x^{x+\delta_+} e^{-W(y)}\,dy=e^{-W(x)}\int_x^{x+\delta_+} e^{W(x)-W(y)}\,dy\\
&\geq e^{-W(x)}\delta_+e^{-\xi(\delta_+)} \geq   Re^{-2W(x)}e^{-2\xi(\delta_+)}
\end{align*}
and
\begin{align*}
\int_x^{S^{-1}(S(x)+R)}e^{-W(y)}\,dy&=\int_x^{x+\delta_+} e^{-W(y)}\,dy=e^{-W(x)}\int_x^{x+\delta_+} e^{W(x)-W(y)}\,dy\\
&\leq e^{-W(x)}\delta_+e^{\xi(\delta_+)} \leq   Re^{-2W(x)}e^{2\xi(\delta_+)}.
\end{align*}

Applying the definition of $\delta_-$ and following the arguments above, we can obtain that
\begin{align*}
Re^{-2W(x)}
e^{-2\xi(\delta_-)}
&\le \int^x_{S^{-1}(S(x)-R)} e^{-W(y)}\,dy=\int^x_{x-\delta_-} e^{-W(y)}\,dy\leq  Re^{-2W(x)}e^{2\xi(\delta_-)}.
\end{align*}

Combining with all the estimates above, we finally arrive at that
\begin{equation}\label{volume}
\begin{split}
2Re^{-2W(x)}e^{-2\max\{\xi(\delta_-),\xi(\delta_+)\}}\leq V(S(x),R)
\leq 2Re^{-2W(x)}e^{2\max\{\xi(\delta_-),\xi(\delta_+)\}}.
\end{split}
\end{equation}

On the other hand, by the Cauchy-Schwartz inequality, we derive
$$
\delta_++\delta_- =\int_{x-\delta_-}^{x+\delta_+}\, dz\leq \left(\int_{x-\delta_-}^{x+\delta_+} e^{-W(z)}\,dz\right)^{1/2}
\cdot\left(\int_{x- \delta_-}^{x+\delta_+} e^{W(z)}\,dz\right)^{1/2}=\left(2RV(S(x),R)\right)^{1/2}.
$$
This along with \eqref{volume} yields that
\begin{align*}
&2Re^{-2W(x)}e^{-2\xi(x,(2RV(S(x),R))^{1/2})}\leq V(S(x),R)
\leq  2Re^{-2W(x)}e^{2\xi(x,(2RV(S(x),R))^{1/2})}.
\end{align*}
This finishes the proof.
\end{proof}

\begin{remark}\label{e:addvd} \rm According to \eqref{volume1} and \eqref{e3-2}, as well as \eqref{e3-3} and \eqref{t1-1-2} below, one can see that the volume $V(S(x),R)$ does not satisfy the so-called volume doubling conditions both for small scale and large scale. By comparing with the known approaches in the literature for diffusions in random media, this is one of main obstacles to establish heat kernel estimates for Brox's diffusion. In order to obtain explicit statements, we will make full use of estimates for the oscillation and the H\"older coefficients of sample paths for Brownian motion. \end{remark}

\section{Quenched heat kernel estimates for small times}\label{section3}

In this section, we will give the proof of Theorem \ref{t1-1}. For every $c_0\ge 1$, we define
\begin{align}\label{e3-3}
\Upsilon(r,c_0;\omega):=\sup_{0\le i \le r}\left(\Xi(i,c_0;\omega)+\Xi(-i,c_0;\omega)\right),\quad \omega\in \Omega,\ r>0,
\end{align} where
$\Xi(x,r;\omega)$ is defined by  \eqref{e3-1a}.

\subsection{On-diagonal  estimates}
The main statement of this part is as follows.
\begin{proposition}\label{c3-1}
There exist  positive constants $C_i$, $1\le i\le 6$, such that
for every $x\in \R$, $t\in (0,1]$ and  almost all $\om\in \Omega$,
\begin{equation}\label{c3-1-1}
\begin{split}
 C_1t^{-1/2} e^{W(x)}\exp\left(-C_2t^{\alpha/2}
\Upsilon\left(1+|x|,C_3;\om\right)
\right)
&\le p^X(t,x,x)\\
&\leq C_4t^{-1/2} e^{W(x)}\exp\left(
C_5t^{\alpha/2}
\Upsilon\left(1+|x|, C_6;\om\right)
\right),
\end{split}
\end{equation}
\end{proposition}

To prove Proposition \ref{c3-1}, we need the following two lemmas.

\begin{lemma}\label{Q-on}
For all $x\in\R$ and $t>0$, it holds almost surely that
\begin{align}\label{p3-1-1}
p^X(t,x,x)
\leq 2\sqrt{2}t^{-1/2}e^{W(x)}e^{3\xi(x,(t/2)^{1/2})},
\end{align}
where $\xi(x, r)$ is defined by \eqref{brownian holder1}.
\end{lemma}
\begin{proof}
According to Lemma \ref{upper} and \eqref{eqxy}, for any  $x\in \R$ and $R>0$,
$$p^X\left(4RV(S(x),R),x,x\right)\leq \frac2{V(S(x),R)}.$$
This along with \eqref{volume1} yields that
$$
p^X\left(4RV(S(x),R),x,x\right)\leq R^{-1}e^{2W(x)}e^{2\xi(x,(2RV(S(x),R))^{1/2})}.
$$
According to \eqref{l2-1-1}, there exists a unique
random variable
$R(x,t)>0$ so that $t=4R(x,t)V(S(x),R(x,t))$. Thus,
\begin{align}\label{p3-1-2}
p^X(t,x,x)\leq  R(x,t) ^{-1}e^{2W(x)}e^{2\xi(x,(t/2)^{1/2})}.
\end{align}

On the other hand, by \eqref{volume1} again,  it holds that
\begin{align*}
t&=4R(x,t)V(S(x),R(x,t))\leq 8R(x,t)^2e^{-2W(x)}e^{2\xi(x,(t/2)^{1/2})},
\end{align*}
which implies
$$R(x,t)^{-1}\leq 2\sqrt{2}t^{-1/2}e^{-W(x)}e^{\xi(x,(t/2)^{1/2})}.$$

Hence, putting this into \eqref{p3-1-2}, we find that
$$p^X(t,x,x)\leq 2\sqrt{2}t^{-1/2}e^{W(x)}e^{3\xi(x,(t/2)^{1/2})}.$$
The proof is complete.
\end{proof}

\begin{lemma}\label{p3-1}
There exist positive constants $C_3, C_4, C_5$
so that for all $t>0$ and $x\in\R$,
\begin{equation}\label{p3-1-4}
p^X\left(t,x,x\right)\ge C_3t^{-1/2}e^{W(x)}e^{-C_4\xi(x,C_5t^{1/2})}.
\end{equation}
\end{lemma}
\begin{proof}
By Lemma \ref{lower} and \eqref{eqxy}, for any $x\in \R$ and $R>0$,
\begin{equation}\label{e:lower1}
p^X\left(2c_2RV(S(x),R),x,x\right)\geq \frac{c_2^2V(S(x),R)^2}{4c_1^2V(S(x),2R)^3},
\end{equation}
where $c_1,c_2$ are positive constants corresponding to $C_1, C_2$ given in \eqref{stopping} and \eqref{stopping1} respectively.

Fix the constant $c_2$ as above. Let $t>0$ and $x\in \R$. According to \eqref{l2-1-1}, we can find a
unique
random variable $R(x,t)>0$ such that
$2c_2R(x,t)V(S(x),2R(x,t))=t$.
Then, by \eqref{volume1}, there is a constant $c_3>0$ such that
\begin{align*}
V(S(x),R(x,t))\ge 2R(x,t)e^{-2W(x)}e^{-2\xi(x,c_3t^{1/2})}
\end{align*} and
\begin{align*}
V(S(x),2R(x,t))&\le
2R(x,t)
e^{-2W(x)}e^{2\xi(x,c_3t^{1/2})}.
\end{align*}
Combining all the estimates above into \eqref{e:lower1} yields that
\begin{align}\label{p3-1-3}
p^X\left(t,x,x\right)\ge c_4R(x,t)^{-1}e^{2W(x)}e^{-10\xi(x,c_3t^{1/2})}.
\end{align}

Furthermore, by \eqref{volume1} again, it holds that
$$
t=2c_2R(x,t)V(S(x),2R(x,t))\ge c_{5}R(x,t)^2e^{-2W(x)}e^{-2\xi(x,c_6t^{1/2})},
$$
which implies
$$R(x,t)^{-1}\geq c_{7}t^{-1/2}e^{-W(x)}e^{-\xi(x,c_6t^{1/2})}.$$
Putting this into \eqref{p3-1-3}, we can prove the desired conclusion \eqref{p3-1-4}.
\end{proof}

\begin{proof}[Proof of Proposition $\ref{c3-1}$] According to \eqref{e3-3},
for any $c_0\ge 1$, we have $\Xi(x,c_0;\omega)\le  \Upsilon\left(1+|x|,c_0;\om\right)$ for all $x\in \R$.
Then, combining \eqref{p3-1-1}, \eqref{p3-1-4} with \eqref{e3-1a} and \eqref{e3-2}, we can prove the desired conclusion \eqref{c3-1-1} immediately.
\end{proof}

\subsection{Off-diagonal estimates}

In this part, we prove off-diagonal quenched estimates of $p^X(t,x,y)$
for small time.

\begin{proposition}\label{p3-2}
There exist
positive constants $C_1$, $C_2$ and $C_3$ such that for every $x,y\in \R$ and $t\in (0,1]$
satisfying that $|x-y|\ge t^{1/2}$,
\begin{equation}\label{p3-2-0}
\begin{split}
p^X(t,x,y)\ge  &C_1e^{W(y)}t^{-1/2}\exp\left(-\frac{C_2|x-y|^2}{t}\right)\\
 &\times\exp\left(-C_{2}t
\Upsilon\left(1+|x|+|y|,C_3;\om\right)^{2/\alpha}
\log\left(t^{1/2}
\Upsilon(1+|x|+|y|, C_3; \om)^{1/\alpha}
\right)\right).\end{split}
\end{equation}
\end{proposition}
\begin{proof}
Note that the process $\{Y(t)\}_{t\ge 0}$ is associated with
the following Dirichlet form  $(\mathcal{E}_Y,\cF_Y)$ on $L^2(\mu_Y)$:
$$\mathcal{E}_Y(f,f)=\frac12\int_\mathbb{R}|f'(x)|^2\,dx,\quad \cF_Y=\{f\in L^2(\mu_Y):\cE(f,f)<\infty\}.$$
According to \cite[(4.17)]{B1} or \cite[(2.4), p. 2997]{CHK},
$$
\cE_Y\left(p^Y(t,x,\cdot),p^Y(t,x,\cdot)\right)\le \frac{p^Y(t,x,x)}{t},\quad x\in \R,\ t>0.
$$
Combining this with \eqref{l2-1-5a} and the symmetry of $p^Y(t,x,y)$ with respect to $(x,y)$ yields that
$$
\left|p^Y(t,x,y)-p^Y(t,y,y)\right|^2\le \frac{2 p^Y(t,y,y)}{t}\cdot|x-y|,\quad x,y\in \R,\ t>0.
$$
Hence, by \eqref{eqxy},
\begin{align*}
\left|p^X(t,x,y)-p^X(t,y,y)\right|&=\left|p^Y\left(t,S(x),S(y)\right)-p^Y\left(t,S(y),S(y)\right)\right|\\
&\le \sqrt{\frac{2p^Y\left(t,S(y),S(y)\right)}{t}\cdot|S(x)-S(y)|}\\
&=\sqrt{\frac{2p^X\left(t,y,y\right)}{t}\cdot|S(x)-S(y)|}
\end{align*}

Furthermore,
for every $x,y\in \R$ with $|x-y|\le 1$,
\begin{align*}
|S(x)-S(y)|&=\left|\int_x^y e^{W(z)}\,dz\right|\le e^{W(y)} \int_x^y e^{|W(z)-W(y)|}\,dz \le e^{W(y)+\xi(y, |x-y|)}|x-y|\\
&\le e^{W(y)+\Xi(y,1)2^\alpha|x-y|^\alpha}|x-y|,
\end{align*}
where in the last inequality we used \eqref{e3-2}.

Therefore, by
\eqref{p3-1-1} and \eqref{p3-1-4},
for all $t\in (0,1]$ and $x,y\in \R$ with $|x-y|\le 1$,
\begin{align*}
   p^X(t,x,y)
&\ge p^X(t,y,y)-\sqrt{\frac{2p^X\left(t,y,y\right)}{t}\cdot|S(x)-S(y)|}\\
&\ge p^X(t,y,y)-\sqrt{\frac{2p^X\left(t,y,y\right)}{t}\cdot e^{W(y)+2^\alpha\Xi(y,1)|x-y|^\alpha}|x-y|}\\
&\ge c_1t^{-1/2}e^{W(y)}\left(e^{-c_2\xi(y,c_3t^{1/2})}-c_4e^{c_5\left(\xi(y,c_6t^{1/2})+\Xi(y,1)|x-y|^\alpha\right)}\sqrt{t^{-1/2}|x-y|}\right)\\
&\ge c_1t^{-1/2}e^{W(y)}\left(e^{-c_2\Xi(y,c_7)t^{\alpha/2}}-c_4e^{c_5\Xi(y,c_7)\left(t^{\alpha/2+|x-y|^\alpha}\right)}\sqrt{t^{-1/2}|x-y|}\right),
\end{align*}
where in the last inequality we used \eqref{e3-2} again and without loss of generality we can take $c_7\ge1$ large enough.
In particular, there exist  positive  constants $c_{8}$, $c_{9}\in (0,1/2)$ such that
\begin{equation}\label{p3-2-2}
p^X(t,x,y)\ge c_{8}t^{-1/2}e^{W(y)}
\end{equation}
for all $ x,y\in \R$ and $t\in (0,1]$ with
$|x-y|\le 2c_{9}t^{1/2}$ and $0<t\le 2c_{9}\Xi(y,c_7)^{-2/\alpha}.$

For $x,y\in \R$ and $t\in (0,1]$, set
\begin{align*}
& N:=N(t,x,y,\omega):=\left[\max\left\{\frac{t\Upsilon(1+|x|+|y|,4c_7)^{2/\alpha}}{c_{9}},\frac{|x-y|^2}{c_{9}^2 t}\right\}
\right]+1,\\
&x_i:=x+\frac{i (y-x)}{N},\quad 0\le i \le N.
\end{align*}
It is easy to verify that for every $u\in B(x_i,\frac{|x-y|}{2N})$ with $1\le i \le N$ and $t\in (0,1]$,
\begin{equation}\label{p3-2-3}
\begin{split}
&\frac{2|x-y|}{N}\le 2c_{9}\left(\frac{t}{N}\right)^{1/2}\le c_7,\\
&\frac{t}{N} \le 2c_{9}\Upsilon(1+|x|+|y|,4c_7)^{-2/\alpha}\le 2c_{9}\Xi(x_i,2c_7)^{-2/\alpha}\le 2c_{9}\Xi(u,c_7)^{-2/\alpha}, \end{split}
\end{equation}
Therefore, for all $0\le i\le N-1$, $u\in B\left(x_{i},\frac{|x-y|}{2N}\right)$ and $v\in B\left(x_{i+1},\frac{|x-y|}{2N}\right)$,
\begin{align*}
p^X\left(\frac{t}{N},u,v\right)&\ge c_{10}\left(\frac{t}{N}\right)^{-1/2}e^{W(v)},\\
&\ge c_{10}\left(\frac{t}{N}\right)^{-1/2}e^{W(x_{i+1})}e^{-|W(v)-W(x_{i+1})|},\\
&\ge c_{10}\left(\frac{t}{N}\right)^{-1/2}e^{W(x_{i+1})}e^{-\Xi(v,c_7)\left(\frac{2|x-y|}{N}\right)^{\alpha}}\\
&\ge c_{11}\left(\frac{t}{N}\right)^{-1/2}e^{W(x_{i+1})},
\end{align*}
where the first inequality follows from \eqref{p3-2-2}, and in the last inequality we used  \eqref{p3-2-3}.
Hence,
\begin{align*}
&p^X(t,x,y)\\
&=\int_{\R}\cdots\int_{\R}p^X\left(\frac{t}{N},x,y_1\right)\prod_{i=1}^{N-2}p^X\left(\frac{t}{N},y_i,y_{i+1}\right)
p^X\left(\frac{t}{N},y_{N-1},y\right)\prod_{i=1}^{N-1}\mu_X(dy_i)\\
&\ge \int_{B\left(x_1,\frac{|x-y|}{2N}\right)}\cdots \int_{B\left(x_{N-1},\frac{|x-y|}{2N}\right)}
p^X\left(\frac{t}{N},x,y_1\right)\prod_{i=1}^{N-2}p^X\left(\frac{t}{N},y_i,y_{i+1}\right)
p^X\left(\frac{t}{N},y_{N-1},y\right)\prod_{i=1}^{N-1}\mu_X(dy_i)\\
&\ge \prod_{i=1}^{N}\left(c_{11}\left(\frac{t}{N}\right)^{-1/2}e^{W(x_i)}\right)\cdot \prod_{i=1}^{N-1}\mu_X\left(B\left(x_i,\frac{|x-y|}{2N}\right)\right).
\end{align*}
On the other hand, it holds that
\begin{align*}
\mu_X\left(B\left(x_i,\frac{|x-y|}{2N}\right)\right)&=
\int_{x_i-\frac{|x-y|}{2N}}^{x_i+\frac{|x-y|}{2N}}e^{-W(z)}\,dz\\
&\ge e^{-W(x_i)}\int_{x_i-\frac{|x-y|}{2N}}^{x_i+\frac{|x-y|}{2N}}e^{-|W(z)-W(x_i)|}\,dz\\
&\ge \frac{|x-y|}{N}e^{-W(x_i)}e^{-\Xi(x_i,c_7)\left(\frac{2|x-y|}{N}\right)^{\alpha}}\\
&\ge c_{12}\frac{|x-y|}{N}e^{-W(x_i)},
\end{align*}
where the last inequality is due to \eqref{p3-2-3} again.

Therefore, combining with all the  estimates above, we find that for every
$t\in(0,1]$ and $x,y\in \R$
with $|x-y|^2\ge t$,
\begin{align*}
&p^X(t,x,y)\\
&\ge \prod_{i=1}^{N}\left(c_{11}\left(\frac{t}{N}\right)^{-1/2}e^{W(x_i)}\right)\cdot \prod_{i=1}^{N-1}\left(c_{12}\frac{|x-y|}{N}e^{-W(x_i)}\right)\\
&\ge c_{13}t^{-1/2}e^{W(y)}\left(c_{14}\frac{|x-y|}{t^{1/2}N^{1/2}}\right)^{N-1}\\
&\ge c_{13}t^{-1/2}e^{W(y)}\left(\min\left\{\frac{c_{15}|x-y|}{t\Upsilon(1+|x|+|y|,4c_7)^{1/\alpha}}, c_{15}\right\}\right)^{N-1}\\
&\ge c_{13}t^{-1/2}e^{W(y)}\\
&\qquad \times \min\left\{\exp\left(-\frac{c_{16}|x-y|^2}{t}\right), \exp\left(-c_{16}t\Upsilon(1+|x|+|y|,4c_7)^{2/\alpha}\log\left(\frac{t\Upsilon(1+|x|+|y|,4c_7)^{1/\alpha}}{|x-y|}\right)\right)\right\}\\
&\ge c_{13}t^{-1/2}e^{W(y)}\exp\left(-\frac{c_{16}|x-y|^2}{t}-c_{16}t\Upsilon(1+|x|+|y|, 4c_7)^{2/\alpha}\log\left(\frac{t\Upsilon(1+|x|+|y|,4c_7)^{1/\alpha}}{|x-y|}\right)\right)\\
&\ge c_{13}t^{-1/2}e^{W(y)}\exp\left(-\frac{c_{16}|x-y|^2}{t}-c_{16}t\Upsilon(1+|x|+|y|,4c_7)^{2/\alpha}\log\left(t^{1/2}\Upsilon(1+|x|+|y|,4c_7)^{1/\alpha}\right)\right),
\end{align*}
where the last inequality follows from $|x-y|^2\ge t$. This proves the  desired assertion.
\end{proof}

To obtain upper bounds of off-diagonal estimates for $p^X(t,x,y)$ for small time, we will make use of the mean of the exit time for the process $\{X(t)\}_{t\ge 0}$.
\begin{lemma}\label{l3-2}
There exist
positive constants $C_4$ and $C_5$ such that for all $x\in \R$ and $R>0$,
\begin{equation}\label{l3-2-0}
C_4e^{-8\Xi(x,R)R^\alpha}R^2\le \bE^x[\tau_{B(x,R)}^X]\le C_5
e^{4\Xi(x,R)R^\alpha}R^2.
\end{equation}
\end{lemma}
\begin{proof}
Since $X(t)=S^{-1}(Y(t))$,  for all $x\in \R$ and $R>0$,
$
\tau^X_{B(x,R)}=\tau^Y_{(S(x-R),S(x+R))}.
$
Hence,
\begin{align*}
\bE^x[\tau_{B(x,R)}^X]&=
\bE^{S(x)}[\tau^Y_{(S(x-R),S(x+R))}]\\
&=\int_{S(x-R)}^{S(x+R)}G^Y_{\left(S(x-R),S(x+R)\right)}\left(S(x),y\right)\,\mu_Y(dy)\\
&=\int_{x-R}^{x+R}G^B_{\left(S(x-R),S(x+R)\right)}\left(S(x),S(z)\right)e^{-W(z)}\,dz,
\end{align*}
where $G_D^Y(x,y)$ and $G_D^B(x,y)$ denote the Green function on the domain $D \subset \R$ associated with
the process $\{Y(t)\}_{t\ge 0}$ and Brownian motion $\{B(t)\}_{t\ge 0}$ respectively, and in the last equality we used the change of variable $y=S(z)$ and the fact that $G_D^Y(x,y)=G_D^B(x,y)$ for every $(x,y)\in D \times D$.

It holds that for every $z\in [x-R,x+R]$,
$$
S(x+R)-S(z)=\int_{z}^{x+R}e^{W(y)}\,dy\ge
e^{W(z)}\int_{z}^{x+R}e^{-|W(y)-W(z)|}\,dy \ge (x+R-z)e^{W(z)}e^{-\xi(x,R)},
$$
where in the last inequality we used  \eqref{brownian holder1}. Similarly,
for every $z\in [x-R,x+R]$,
\begin{align*}
S(x+R)-S(z)&=\int_{z}^{x+R}e^{W(y)}\,dy\le
e^{W(z)}\int_{z}^{x+R}e^{|W(y)-W(z)|}\,dy \le (x+R-z)e^{W(z)}e^{\xi(x,R)}.
\end{align*}
 By the same way, we can prove that for all $x\in \R$ and $z\in [x-R,x+R]$,
$$
(z-x+R)e^{W(z)}e^{-\xi(x,R)}\le S(z)-S(x-R)\le (z-x+R)e^{W(z)}e^{\xi(x,R)}.
$$
On the other hand, according to \cite[p.\ 45]{CZ},   for every $a<b$,
$$
G^B_{(a,b)}(y,z)=
\begin{cases}
 2(b-a)^{-1}\left(y-a\right)\left(b-z\right),&\quad a<y<z<b,\\
 2(b-a)^{-1}\left(b-y\right)\left(z-a\right),&\quad a<z\le y<b.
\end{cases}
$$
Therefore, for every $x\in \R$ and $R>0$,
\begin{align*}
\bE^x[\tau_{B(x,R)}^X]& =\int_{x-R}^{x+R}G^B_{\left(S(x-R),S(x+R)\right)}\left(S(x),S(z)\right)e^{-W(z)}\,dz\\
&\ge \frac{2}{S(x+R)-S(x-R)}\int_{x}^{x+R}\left(S(x)-S(x-R)\right)\left(S(x+R)-S(z)\right)e^{-W(z)}\,dz\\
&\ge e^{-3\xi(x,R)}\int_x^{x+R}\left(x+R-z\right)e^{W(x)-W(z)}dz
\ge c_1R^2e^{-4\xi(x,R)}\ge c_1R^2e^{-8\Xi(x,R)R^\alpha},
\end{align*}
where in the third inequality we have used
\begin{align*}
e^{W(x)-W(z)}\ge e^{-|W(x)-W(z)|}\ge e^{-\xi(x,R)},\quad z\in (x,x+R),
\end{align*}
and the last inequality follows from \eqref{e3-2}.
Thus, we prove the first inequality in \eqref{l3-2-0}.

Meanwhile, according to all the estimates above, we find that for all $x\in \R$ and $R>0$,
\begin{align*}
 \bE^x [\tau_{B(x,R)}^X ]
&=\int_{x-R}^{x+R}G^B_{\left(S(x-R),S(x+R)\right)}\left(S(x),S(z)\right)e^{-W(z)}\,dz\\
&\le c_2\int_{x-R}^{x+R}\max\left\{\left|S(x+R)-S(x)\right|, \left|S(x-R)-S(x)\right|\right\}e^{-W(z)}\,dz\\
&\le c_2R\int_{x-R}^{x+R}e^{W(x)-W(z)}e^{\xi(x,R)}\,dz\le
c_3R^2e^{4\Xi(x,R)R^\alpha},
\end{align*}
where in the second inequality we have used the fact that
$$
\max\left\{|S(x+R)-S(x)|,|S(x-R)-S(x)|\right\}\le Re^{W(x)}e^{\xi(x,R)},
$$
and the last inequality follows from \eqref{e3-2}. This proves the second inequality in \eqref{l3-2-0}.
\end{proof}

\begin{lemma}\label{l3-3}
There exist constants $C_6\in (0,1/2)$ and $C_7>0$ such that for all $x\in \R$, $t>0$ and $R>0$,
\begin{equation}\label{l3-3-1}
\bP^x(\tau^X_{B(x,R)}\le t)
\le 1-C_6e^{-16\Xi(x,4R)R^\alpha}+\frac{C_7 e^{-8\Xi(x,4R)R^\alpha}t}{R^2}.
\end{equation}
\end{lemma}
\begin{proof}
According to \eqref{l3-2-0},  for every $x\in \R$ and $R>0$,
\begin{align*}
c_1 e^{-8\Xi(x,R)R^\alpha}R^2\leq& \bE^x[\tau^X_{B(x,R)}]\leq t+\bE^x [\I_{\{\tau^X_{B(x,R)}>t\}}\bE^{X(t)}[\tau^X_{B(x,R)}] ]
\\\leq&t+c_2e^{8\Xi(x,4R)R^\alpha}R^2\bP^x(\tau^X_{B(x,R)}>t).
\end{align*}
Here in the last inequality we have used the fact
\begin{align*}
\bE^z[\tau^X_{B(x,R)}]\le \bE^z[\tau^X_{B(z,2R)}]\le
c_2e^{4\Xi(z,2R)(2R)^\alpha}R^2\le c_2e^{8\Xi(x,4R)R^\alpha}R^2,
\quad z\in B(x,R).
\end{align*}
Hence,
$$
\bP^x(\tau^X_{B(x,R)}>t)\geq \frac{c_1e^{-8\Xi(x,R)R^\alpha}R^2-t}{c_2 e^{8\Xi(x,4R)R^\alpha}R^2}=\frac{c_1e^{-16\Xi(x,4R)R^\alpha}}{c_2}
-\frac{te^{-8\Xi(x,4R)R^\alpha}}{c_2R^2}.
$$
So \eqref{l3-3-1} is true.
\end{proof}

We also need the following elementary lemma, see \cite[Lemma 3.6]{B} or \cite[Lemma 1.1]{BB}.

\begin{lemma}\label{l2-5}
Let $\eta_1,\eta_2,\cdots,\eta_n$ and $\Phi$ be non-negative random variables on  some probability space
$(\Omega_0,\mathscr{F}_0,\mathbf{P})$
such that $\Phi\ge \sum_{i=1}^n \eta_i$. Suppose that there exist positive constants
$a\in (0,1)$ and $b>0$ such that
$$
\mathbf{P}\left(\eta_i\le t|\sigma\{\eta_1,\cdots,\eta_{i-1}\}\right)\le a+bt,\quad t>0,\ 1\le i \le n,
$$ where $\eta_0=0$.
Then
$$
\log \bP\left(\Phi\le t\right)\le 2\left(\frac{bnt}{a}\right)^{1/2}-n\log \frac{1}{a},\quad t>0.
$$
\end{lemma}

\begin{lemma}\label{l3-5}
There exist  positive constants $C_i$, $8\le i\le 11$, such that for every $x\in \R$, $t\in  (0,1]$ and $R>0$ with
$R^2\ge C_8 t$,
\begin{equation}\label{l3-5-1}
\bP^x\left(\tau_{B(x,R)}^X\le t\right)\le \exp\left(-\frac{C_9R^2}{t}+C_{10}t^{1/2}R^{1/2}\Upsilon(1+R+|x|, C_{11};\omega)^{{3}/({2\alpha})}\right)
\end{equation}
\end{lemma}
\begin{proof}
According to \eqref{l3-3-1} and   \eqref{e3-3}, there are
positive constants $c_1\in (0,1)$ and $c_2, c_3$ such that
\begin{equation}\label{l3-5-2a}
\bP^z (\tau^X_{B(z,r)}\le t )\le c_1+\frac{c_2t}{r^2},\quad z\in B(0,1+|x|),\ r\le \min\{\Upsilon(1+|x|,c_3)^{-1/\alpha},1\}.
\end{equation}
Throughout the proof below, we always write $\Upsilon(1+|x|,c_3)$ as $\Upsilon(1+|x|)$ for the simplicity of notation.

Let $R^2\ge C_8 t$, and set
\begin{equation*}
\begin{split}
N=N(R,t):=\max\left\{\left[\frac{c_0R^2}{t}\right],\left[R\Upsilon(1+R+|x|)^{1/\alpha}\right]\right\}+1,
\end{split}
\end{equation*}
where $c_0$ and $C_8$ are positive constants  to be determined later.
Define
\begin{align*}
\tau_0:=0,\,\, \tau_{i}:=\inf \{t>\tau_{i-1}: |X(\tau_{i-1})-X(t) |=R/{N} \},\,\, \eta_i:=\tau_i-\tau_{i-1},\quad 1\le i \le N.
\end{align*}
Note that under $\bP^x$
\begin{align*}
|X(\tau_{N})-x|&\le \sum_{i=1}^{N}\left|X(\tau_{i})-X(\tau_{i-1})\right|\le \sum_{i=1}^N \frac{R}{N}=R,
\end{align*}
so   $\tau_{B(x,R)}^X\ge \tau_N=\sum_{i=1}^N \eta_i$.
Then, by the strong Markov property of the process $\{X(t)\}_{t\ge 0}$,
for every $1\le i \le N$ and $t>0$,
\begin{align*}
\bP^x\left(\eta_i\le t|\sigma\left\{\eta_1,\cdots,\eta_{i-1}\right\}\right)&=\bE^x\left[\bP^{X(\tau_{i-1})}\left(\tau_{B\left(X(\tau_{i-1}),\frac{R}{N}\right)}^X\le t\right)\right]\\
&\le \sup_{y\in B(x,R)}\bP^y\left(\tau_{B\left(y,\frac{R}{N}\right)}^X\le t\right)\\
&\le c_1+\frac{c_2N^2t}{R^2}.
\end{align*}
Here the last inequality follows from
\eqref{l3-5-2a} and
the fact that
\begin{align*}
\frac{R}{N}
&\le \min\left\{\Upsilon(1+|x|+R)^{-1/\alpha},\frac{t}{c_0R}\right\}
\le \min\left\{\Upsilon(1+|x|+R)^{-1/\alpha},\frac{t^{1/2}}{C_8^{1/2}c_0}\right\}\\
&
\le \min\left\{\Upsilon(1+|x|+R)^{-1/\alpha},1\right\},
\end{align*} where in the last equality we take $C_8=c_0^{-2}$.

Therefore, applying Lemma \ref{l2-5} with $n=N$, $a=c_1$ and $b=\frac{c_2N^2}{R^2}$, we derive that
\begin{equation}\label{l3-5-2}
\log \bP^x(\tau^X_{B(x,R)}\le t)
 \le \frac{c_4N^{3/2}t^{1/2}}{R}-N\log\left(\frac{1}{c_1}\right)\le -N\left(c_5-\frac{c_4N^{1/2}t^{1/2}}{R}\right).\\
\end{equation}
Now, choose $c_0=\frac{1}{4}\left(\frac{c_5}{c_4}\right)^2$ in the definition of $N$,  and we deduce that
\begin{align*}
\frac{c_4N^{1/2}t^{1/2}}{R} -c_5\le
\begin{cases}
-\frac{c_5}{2}\ &{\rm if}\  \frac{c_0R^2}{t}\le R\Upsilon(1+R+|x|)^{1/\alpha},\\
\frac{c_6t^{1/2}\Upsilon(1+R+|x|)^{\frac{1}{2\alpha}}}{R^{1/2}}-c_5\ &{\rm if}\ \frac{c_0R^2}{t}\ge R\Upsilon(1+R+|x|)^{1/\alpha}.
\end{cases}
\end{align*}
Hence, putting this into \eqref{l3-5-2}, we can prove the  conclusion \eqref{l3-5-1}.
\end{proof}

\begin{proposition}\label{p3-3}
There exist   positive constants $C_{i}$, $12\le i \le 15$, such that
for every $x,y\in \R$ and $t>0$ with $t\le |x-y|^2$ and for almost surely all $\om\in \Omega$,
\begin{equation}\label{p3-3-1}
  p^X(t,x,y)\le C_{12}t^{-1/2}e^{W(y)}
\exp\Big(-\frac{C_{13}|x-y|^2}{2t}+C_{14}t^{3}\Upsilon(1+|x|+|y|,C_{15};\om)^{{6}/{\alpha}}\Big).
\end{equation}
\end{proposition}
\begin{proof}
By the symmetry property of $p^X(t,x,y)$ with respect to $(x,y)$, we can assume that $x<y$. Set
\begin{align*}
p^X(t,x,y)
&=\bE^x\left[p^X\left(\frac{t}{2},X\left(\frac{t}{2}\right),y\right)\right]\\
&=\bE^x\left[p^X\left(\frac{t}{2},X\left(\frac{t}{2}\right),y\right)\I_{\{X\left(\frac{t}{2}\right)\ge \frac{x+y}{2}\}}\right]
+\bE^x\left[p^X\left(\frac{t}{2},X\left(\frac{t}{2}\right),y\right)\I_{\{X\left(\frac{t}{2}\right)< \frac{x+y}{2}\}}\right]\\
&=:I_1+I_2.
\end{align*}

Let $D:=\{z\in \R:|z-x|\le |z-y|\}$. According to the strong Markov property of the process $\{X(t)\}_{t\ge 0}$, we obtain
\begin{align*}
I_1&\le \bE^x\left[p^X\left(\frac{t}{2},X\left(\frac{t}{2}\right),y\right)\I_{\{\tau_D^X\le t/2\}}\right]\\
&=\bE^x\left[\bE^{X(\tau_D^X)}\left[p^X\left(\frac{t}{2},X\left(\frac{t}{2}-\tau_D^X\right),y\right)\right]\I_{\{\tau_D^X\le t/2\}}\right]\\
&=\bE^x\left[\bE^{\frac{x+y}{2}}\left[p^X\left(\frac{t}{2},X\left(\frac{t}{2}-\tau_D^X\right),y\right)\right]\I_{\{\tau_D^X\le t/2\}}\right]\\
&\le \bP^x(\tau_{B(x,|x-y|/2)}^X\le t/2)\cdot\sup_{s\in [0,t/2]}\bE^{\frac{x+y}{2}}\left[p^X\left(\frac{t}{2},X\left(\frac{t}{2}-s\right),y\right)\right]\\
&=\bP^x(\tau_{B(x,|x-y|/2)}^X\le t/2)\cdot\sup_{s\in [0,t/2]}p^X\left(t-s,\frac{x+y}{2},y\right).
\end{align*}
Here in the second equality we  used the fact $X(\tau_D^X)=\frac{x+y}{2}$ since we assume $x<y$,  the second inequality follows from
the fact that $B(x,|x-y|/2)\subset D$, and in the last equality we used the semigroup property of the heat kernel $p^X(t,x,y)$.

According to \eqref{l3-5-1}, it holds that for all $x,y\in \R^d$ and $t>0$ with $|x-y|^2\ge t$,
\begin{align*}
\bP^x\left(\tau_{B(x,|x-y|/2)}^X\le t/2\right)\le c_0\exp\left(-\frac{c_1|x-y|^2}{t}+c_2t^{1/2}|x-y|^{1/2}\Upsilon(1+|x|+|y|,c_3)^{{3}/({2\alpha})}\right).
\end{align*} Indeed, by \eqref{l3-5-1}, the estimate above holds with $c_0=1$ when $(|x-y|/2)^2\ge C_8t;$ when $t\le |x-y|^2\le 4C_8 t$, the estimate above still holds by taking $c_0\ge1$ large enough.
On the other hand,
\begin{align*}
& \sup_{s\in [0,t/2]}p^X\left(t-s,\frac{x+y}{2},y\right)\\
&\le \sup_{s\in [0,t/2]}\left(p^X\left(t-s,\frac{x+y}{2},\frac{x+y}{2}\right)\right)^{1/2}
\cdot \left(p^X\left(t-s,y,y\right)\right)^{1/2}\\
&\le \left(p^X\left(\frac{t}{2},\frac{x+y}{2},\frac{x+y}{2}\right)\right)^{1/2}
\cdot \left(p^X\left(\frac{t}{2},y,y\right)\right)^{1/2}\\
&\le c_4t^{-1/2}\exp\left(\frac{W\left(\frac{x+y}{2}\right)+W(y)}{2}\right)
\cdot \exp\left(c_4t^{\alpha/2}\Upsilon(1+|x|+|y|,c_4)\right).
\end{align*}
Here the second inequality follows from the fact that $t\mapsto p(t,x,x)$ is non-increasing for
every fixed $x\in \R$, and in the last inequality we used \eqref{c3-1-1}.

Putting all the estimates above together, we have
\begin{align*}
I_1&\le c_5t^{-1/2}
\cdot \exp\left(-\frac{c_6|x-y|^2}{t}\right)\exp\left(\frac{W\left(\frac{x+y}{2}\right)+W(y)}{2}\right)\\
&\quad \times  \exp\left(c_5
t^{1/2}|x-y|^{1/2}\Upsilon(1+|x|+|y|,c_5)^{ {3}/({2\alpha})}+c_5t^{\alpha/2}\Upsilon(1+|x|+|y|,c_5)\right).
\end{align*}

According to \eqref{e3-2} and  \eqref{e3-3}, we can obtain that
\begin{align*}
\left|W\left(\frac{x+y}{2}\right)-W(y)\right|&\le \sum_{i=0}^{N-1} \left|W(y_{i+1})-W(y_i)\right|
\le c_7\sum_{i=0}^{N-1}\Xi(y_i,1)|y_{i+1}-y_i|^{\alpha}\\&
\le c_8\max\{|x-y|,|x-y|^\alpha\}\Upsilon(1+|x|+|y|,2),
\end{align*}
where $y_0:=\frac{x+y}{2}$, $y_i:=y_0+\frac{i(y-x)}{2N}$ for $1\le i \le N-1$ with $N:=\left[\frac{y-x}{2}\right]+1$, and
$y_N:=y$.

Hence,
\begin{align*}
 I_1
 &\le c_5t^{-1/2}e^{W(y)}
\cdot \exp\Big(-\frac{c_6|x-y|^2}{t}+c_9t^{1/2}|x-y|^{1/2}\Upsilon(1+|x|+|y|,c_9)^{ {3}/({2\alpha})}\\
&\quad\quad\quad\quad \quad\quad\quad \quad\quad\quad+c_9\left(t^{\alpha/2}+\max\{|x-y|,|x-y|^\alpha\}\right)
 \Upsilon(1+|x|+|y|,c_9)\Big)\\
&\le c_{10}t^{-1/2}e^{W(y)}
\exp\Big(-\frac{c_6|x-y|^2}{2t}+
c_{10}t^{3}
\Upsilon(1+|x|+|y|,c_9)^{ {6}/{\alpha}}\Big).
\end{align*}
Here we have used the property that (due to $|x-y|^2\ge t$)
\begin{align*}
t^{\alpha/2} \Upsilon(1+|x|+|y|,c_9)\le |x-y|^{\alpha}\Upsilon(1+|x|+|y|,c_9)\le c_{11}\left(1+|x-y|^{ {3}/{2}}\Upsilon(1+|x|+|y|,c_9)^{ {3}/({2\alpha})}\right),
\end{align*}
and for every $\varepsilon>0$ there exists a positive constant $c_{12}(\varepsilon)$ so that
\begin{align*}
t^{1/2}|x-y|^{1/2}\Upsilon(1+|x|+|y|,c_9)^{ {3}/({2\alpha})}&\le |x-y|^{{3}/{2}}\Upsilon(1+|x|+|y|,c_9)^{ {3}/({2\alpha})}\\
&\le \frac{\varepsilon|x-y|^2}{t}+c_{12}(\varepsilon)t^3\Upsilon(1+|x|+|y|,c_9)^{ {6}/{\alpha}}
\end{align*}
 and \begin{align*}
|x-y|\Upsilon(1+|x|+|y|,c_9)&\le \frac{\varepsilon|x-y|^2}{t}+c_{12}(\varepsilon)t\Upsilon(1+|x|+|y|,c_9)^{2}\\
&\le \frac{\varepsilon|x-y|^2}{t}+ c_{13}(\varepsilon)\left(1+t^{3/\alpha}\Upsilon(1+|x|+|y|,c_9)^{6/\alpha}\right).
\end{align*}

By the symmetry of $p^X(t,x,y)$, we have
$$I_2=\bE^y\left[p^X\left(\frac{t}{2},X\left(\frac{t}{2}\right),x\right)\I_{\{X\left(\frac{t}{2}\right)<\frac{x+y}{2}\}}\right].$$
Using the expression above and applying the same argument as that for $I_1$ (in particular, changing the position of $x$ and $y$), we can obtain
$$
  I_2\le   c_{14}t^{-1/2}e^{W(y)}
\exp\Big(-\frac{c_{15}|x-y|^2}{2t}+c_{14}t^{3}\Upsilon(1+|x|+|y|)^{ {6}/{\alpha}}\Big).\\
$$

Therefore, according to both estimates for $I_1$ and $I_2$, we can obtain the desired conclusion \eqref{p3-3-1}.
\end{proof}

Now, we are in a position to present the
\begin{proof}[Proof of Theorem $\ref{t1-1}$]

Given $\alpha\in (0,1/2)$, $x\in \R$ and $r>0$, let $$\|f\|_{x,r,\alpha}:=\sup_{s,t\in [x-r,x+r]}\frac{\left|f(s)-f(t)\right|}{|t-s|^{\alpha}}$$ for every
$f\in C([x-r,x+r]; \R)$ such that $f(x)=0$.

Recall that
\begin{align*}
\Xi(x,r;\omega):=\sup_{s,t\in [x-r,x+r]}\frac{\left|W(s)-W(t)\right|}{|t-s|^{\alpha}}=
\|W(\cdot)-W(x)\|_{x,r,\alpha},\quad  \omega\in \Omega,\ x\in \R,\ r>0.
\end{align*}
According to Fernique's theorem (see e.g. \cite[p.\ 159--160]{Ku} or \cite[Theorem 1.2]{FO}) and the stationary property
of $\{W(t)-W(s)\}_{s,t\in \R}$,
for any $r>0$,
we can find
a positive constant $\lambda(r)$ (which only depends on $r$) such that
\begin{align}\label{t1-1-0}
\sup_{x\in \R}\Ee\left[\exp\left(\lambda(r)|\Xi(x,r;\omega)|^2\right)\right]<\infty.
\end{align}
Using \eqref{t1-1-0} and \eqref{e3-3}, and following the argument for \eqref{l2-1-3}, we can prove that, for any $C>0$, there exist a random variable $R_0(\om)>0$ and a constant $c_1>0$ such that
\begin{align}\label{t1-1-1}
|\Upsilon(R, C;\omega)|\le c_1\sqrt{\log(1+|R|)},\quad \omega\in \Omega,\ R>R_0(\omega).
\end{align}
The estimate above also implies that there is a random variable $c_2(\om)>0$ such that
\begin{align}\label{t1-1-2}
|\Upsilon(R, C;\omega)|\le c_2(\omega)\sqrt{\log(1+|R|)},\quad \omega\in \Omega,\ R>0.
\end{align}

Combining \eqref{t1-1-1}, \eqref{t1-1-2} with \eqref{c3-1-1}, \eqref{p3-2-0} and \eqref{p3-3-1}, we then prove the desired two-sided estimates
for $p(t,x,y,\om)$.
\end{proof}

\begin{proof}[Proof of Corollary $\ref{c1-1}$]
For every $c_1>0$, there exists a positive constant $c_2$ such that for all $x\in \R$ and $t\in (0,1]$,
$$
\frac{|x|^2}{t}\ge -c_2+c_1\max\{t^2[\log(2+|x|)]^{2/\alpha}, t^3[\log(2+|x|)]^{3/\alpha}\}.
$$
Combining this with Theorem  \ref{t1-1}, we can obtain the desired conclusion
immediately.
\end{proof}

\section{Annealed heat kernel estimates for large times}\label{section4}
This section is devoted to the proof of Theorem \ref{thm22}. For simplicity, we only prove the assertion for the case that $x=0$. In particular, in this case, $p(t,0,0)=p^X(t,0,0)$ since $W(0,\omega)=0$.

Throughout the proof, for every $x\in \R$, $R>0$ and $\om\in \Omega$, let $V(x,R,\om)$, $V_+(x,R,\om)$, $V_-(x,R,\om)$, $\xi(x, r,\omega)$, $\delta_+(x,R,\om)$ and
$\delta_-(x,R,\om)$ be those defined in previous sections. When $x=0$, for simplicity we write $V(x,R,\om)$, $V_+(x,R,\om)$, $V_-(x,R,\om)$, $\xi(x, r, \omega)$, $\delta_+(x,R,\om)$ and
$\delta_-(x,R,\om)$ as $V(R,\om)$, $V_+(R,\om)$, $V_-(R,\om)$, $\xi(r, \omega)$, $\delta_+(R,\om)$ and $\delta_-(R,\om)$ respectively.
Due to \eqref{l2-1-1}, for every $t>0$ and $\om\in \Omega$, we can define
$R(t,\om)$, $R_+(t,\om)$ and  $R_-(t,\om)$ to be the unique elements in $(0,\infty)$ such that
\begin{equation}\label{e4-1}
 4R(t,\om)V\left(R(t,\om),\om\right)= 4R_+(t,\om)V_+\left(R_+(t,\om),\om\right)=4R_-(t,\om)V_-\left(R_-(t,\om),\om\right)=t.
 \end{equation}

\subsection{Upper bound}
In this part we will prove the following statement.
\begin{proposition}\label{T:UP}
Let $\alpha\in (0,1/2)$ be the constant in \eqref{e3-1a}. Then,
there are constants $C_1>0$ and $T_1>0$ so that for all $t\ge T_1$,
$$\Ee\left[p(t,0,0)\right]\leq \frac{C_1(\log\log t)^{4+1/(2\alpha)}}{\log^2t}.$$
\end{proposition}

For any
$x\in \R$ and $\omega\in \Omega$, set $\Xi(x,\omega):=\Xi(x,1,\omega)$, i.e.,
$$\Xi(x,\omega):=\sup_{s,t\in [x-1,x+1]}\frac{|W(s)-W(t)|}{|s-t|^\alpha}.$$ For simplicity, we write $\Xi(0,\omega)$ as $\Xi(\omega)$.
As explained before, by Fernique's theorem, there exists a constant $\lambda>0$ so that
\begin{equation}\label{ieqR}\Ee[e^{\lambda\Xi^2}]<\infty.\end{equation} On the other hand, according to \eqref{e3-2}, for any $r\in (0,1]$,
$$\xi(r):=\xi(0,r,\om)=\sup_{-r\le s,t\le r}|W(s)-W(t)|\le (2r)^{\alpha}\sup_{-r\le s,t\le r}\frac{|W(s)-W(t)|}{|s-t|^\alpha}\le (2r)^\alpha \Xi(\omega).$$
Thus, by Lemma \ref{Q-on}, for all $t\in (0,1]$, it holds almost surely that
\begin{align}\label{p3-1-1--}
p^X(t,0,0)
\leq 2\sqrt{2}t^{-1/2} e^{3\xi((t/2)^{1/2})}\le  2\sqrt{2}t^{-1/2}e^{c_0t^{\alpha/2}\Xi(\omega)}.
\end{align}

We begin with the following lemma.
\bl\label{small}
Assume that there exist constants $C_2,\theta_1>0$ and $\theta_2\in\R$ such that for all
$t\ge4$,
there is $\Omega_t\subset \Omega$ so that
\begin{equation}\label{l4-1-0}
\Pp(\Omega_t)\leq \frac{C_2(\log\log t)^{\theta_2}} {\log^{\theta_1} t}.
\end{equation}
Then there is a constant $C_3>0$ such that for all $t\ge 4$,
\begin{equation}\label{l4-1-0a}
\Ee[p^X(t,0,0)\I_{\Omega_t}]\leq
\frac{C_3(\log\log t)^{\theta_2+{1}/{(2\alpha)}}}{\log^{\theta_1}t}.
\end{equation}
\el
\begin{proof}
For any $K\geq 1$, set $\Lambda_K:=\{\omega\in \Omega:\Xi(\omega)\leq K\}$.
According to \eqref{p3-1-1--}, there is a constant $c_1>0$ such that
\begin{align}\label{kheat}
p^X\left(\frac{1}{K^{2/\alpha}},0,0,\omega\right)\leq  c_1K^{1/\alpha},\quad
\omega\in \Lambda_K.
\end{align}

Next, we choose $K_0>1$. In particular, $\frac{1}{K_0^{2/\alpha}}\le 1$. By \eqref{kheat} and the fact $t\mapsto p^X(t,0,0)$ is decreasing (which can be verified by the same  argument as that for $p^Y(t,0,0)$ as in the proof of Lemma \ref{l2-1}),
we get
$$
\sup_{t\ge 1}p^X\left(t,0,0,\omega\right)\le p^X\left(\frac{1}{K_0^{2/\alpha}},0,0,\omega\right)\le c_3,\quad \omega\in \Lambda_{K_0}.
$$
Furthermore, set
$$\Lambda_{K_0}^{-1}:=\{\omega\in \Omega: K_0< \Xi(\omega)\le K_0+1\}$$ and
$$\Lambda_{K_0}^k:=\{\omega\in \Omega: K_0+2^k< \Xi(\omega)\le K_0+2^{k+1}\},\quad k\ge 0.$$
By \eqref{ieqR},
\begin{equation}\label{e:4.8}
\Pp(\Lambda_{K_0}^k)\le c_4e^{-\lambda(K_0+2^k)^2}\le
c_5\exp\left(-c_62^{2k}\right),
\quad k\ge 0.
\end{equation} We note that, by adjusting the constants $c_5$ and $c_6$, one can see that the estimate above also holds with $k=-1$.
Using \eqref{kheat} and the fact $t\mapsto p^X(t,0,0)$ is decreasing again, we know that for every $k\ge -1$,
\begin{equation}\label{e:4.9}
\sup_{t\ge 1}p^X\left(t,0,0,\omega\right)\le p^X\left(\frac{1}{(K_0+2^{k+1})^{2/\alpha}},0,0,\om\right)\le c_7(K_0+2^{k+1})^{1/\alpha}
\le c_82^{k/\alpha},\quad \omega\in \Lambda_{K_0}^k.
\end{equation}

Therefore, by \eqref{e:4.9}, \eqref{e:4.8} and \eqref{l4-1-0},  we obtain for every
$t\ge4$,
\begin{align*}
\Ee\left[p^X(t,0,0)\I_{\Omega_t}\right]&=\Ee\left[p^X(t,0,0)\I_{\Omega_t\cap \Lambda_{K_0}}\right]+
\sum_{k=-1}^\infty\Ee\left[p^X(t,0,0)\I_{\Omega_t\cap \Lambda_{K_0}^k}\right]\\
&\le c_9\Pp(\Omega_t)+c_9\sum_{k=-1}^\infty 2^{k/\alpha}\min\{\Pp(\Omega_t),\Pp(\Lambda_{K_0^k})\}\\
&\le \frac{c_{10}(\log\log t)^{\theta_2}}{\log^{\theta_1}t}+
c_{10}\left(\sum_{k=-1}^{k_0}2^{k/\alpha}\frac{(\log\log t)^{\theta_2}}{\log^{\theta_1}t}
+\sum_{k=k_0+1}^\infty 2^{k/\alpha}
\exp (-c_62^{2k} )
\right)\\
&\le \frac{c_{11}(\log\log t)^{\theta_2+{1}/{(2\alpha)}}}{\log^{\theta_1}t},
\end{align*}
where in the second inequality
$$
k_0:=\sup\left\{k\in \N_+: \exp(-c_62^{2k})\ge\frac{(\log\log t)^{\theta_2}}{\log^{\theta_1}t},\right\}
$$
and the last inequality follows from
$
2^{k_0/\alpha}\le c_{12}(\log\log t)^{1/(2\alpha)}.
$
The proof is complete.
\end{proof}

We next give a decomposition of the probability space to give an analysis on
the valley of $W(\cdot,\om)$, which is a key ingredient to the proof of Proposition \ref{T:UP}. Define
\begin{align*}
&H_{a,b}(\omega):=\inf\{t>0: W(t,\omega)\notin(a,b)\},\quad H_z(\omega):=\inf\{t>0: W(t,\omega)=z\},\\
&\tilde H_{a,b}(\omega):=
\sup
\{t<0: W(t,\omega)\notin(a,b)\},\quad \tilde H_z(\omega):=\sup\{t<0: W(t,\omega)=z\}.
\end{align*}
Given
$R>0$, set
\begin{align*}
&\xi_{+,R}(\omega):=\left\{\omega\in \Omega: \sup_{z_1,z_2\in [0,R]  \text { with } |z_1-z_2|\le 1}\frac{|W(z_1,\omega)-W(z_2,\omega)|}{|z_1-z_2|^\alpha}\right\},\\
&\xi_{-,R}(\omega):=\left\{\omega\in \Omega: \sup_{z_1,z_2\in [-R,0] \text { with }|z_1-z_2|\le 1}\frac{|W(z_1,\omega)-W(z_2,\omega)|}{|z_1-z_2|^\alpha}\right\}.
\end{align*}
Fix $\theta\in (0,1/2)$, and let $K_1,K_2$ (which are large enough) be positive constants to be determined later. We define
\begin{align*}
&\Lambda_t^1:=\Big\{\omega\in \Omega: H_{-K_1\log\log t}(\omega)\le H_{\theta\log t}(\omega),\  H_{2\theta\log t}(\omega)\le H_{-(1-2\theta)\log t}(\omega),\\
&\qquad\qquad\qquad\quad\ \xi_{+,\log^4 t}(\omega)\le K_2\sqrt{\log\log t},
H_{2\theta \log t}(\omega)\le \log^4 t\Big\},\\
&\Lambda_t^2:=\Big\{\omega\in \Omega: H_{2\theta\log t}(\omega)> H_{-(1-2\theta)\log t}(\omega),\ \xi_{+,\log^4 t}(\omega)\le K_2\sqrt{\log\log t},\ H_{2\theta \log t}(\omega)\le \log^4 t\Big\},\\
&\Lambda_t^3:=\left\{\omega\in \Omega: H_{2\theta \log t}(\omega)>\log^4 t\right\},\\
&\Lambda_t^4:=\{\omega\in \Omega:
\xi_{+,\log^4 t}(\omega)> K_2\sqrt{\log\log t}\},\\
&\Lambda_t^5:=\left\{\omega\in \Omega: H_{-K_1\log\log t}(\omega)> H_{\theta\log t}(\omega)\right\}.\
\end{align*}
Furthermore, define $\tilde \Lambda_t^i$, $i=1,\cdots,5$, by the same way as above with $\tilde H_z(\omega)$ and $\xi_{-,R}(\omega)$ instead of $H_z(\omega)$ and $\xi_{+,R}(\omega)$ respectively,
which represent the same event of $\{W(x)\}_{x\in \R}$ at $(-\infty,0]$.
Obviously,
\begin{align*}
\Omega=\cup_{i=1}^5 \Lambda_t^i=\cup_{i=1}^5 \tilde \Lambda_t^i.
\end{align*}

\begin{lemma}\label{l4-2}
For any $K_2>0$, there exist positive constants $K_1^*$, $T_2$ and $C_4$ such that, if $K_1\ge K_1^*$ in the definitions of $\Lambda_t^1$ and
$\tilde \Lambda_t^1$ above, then for any $t\ge T_2$,
\begin{equation}\label{l4-2-1}
\Ee\left[p(t,0,0)\I_{\Lambda_t^1}\right]\le \frac{C_4}{\log^2 t},\quad
\Ee\left[p(t,0,0)\I_{\tilde \Lambda_t^1}\right]\le \frac{C_4}{\log^2 t}.
\end{equation}
\end{lemma}
\begin{proof}
We only prove the first inequality in \eqref{l4-2-1}, and the second one can be proved by
exactly the same argument.
For every $\omega\in \Lambda_t^1$, set
$$r_1(\omega):=\sup_{0\le z \le H_{\theta\log t}(\omega)}\left(-W(z,\omega)\right),\quad r_2(\omega):=\sup_{H_{\theta\log t}(\om)\le z\le H_{2\theta\log t}(\omega)}\left(-W(z,\omega)\right).
$$
By the definition of $\Lambda_t^1$,
\begin{equation}\label{l4-2-2}
K_1\log\log t\le r_1(\omega)\le (1-2\theta)\log t,\,\, \,\, r_2(\omega)\le (1-2\theta)\log t,\quad \omega\in \Lambda_t^1,
\end{equation}
where we used the facts that
\begin{align*}
&H_{-K_1\log\log t}(\omega)\le H_{\theta\log t}(\omega)\quad\hbox{ if and only if }\quad \sup_{0\le z\le H_{\theta\log t}(\omega)}
\left(-W(z,\omega)\right)\ge K_1\log\log t,\\
&H_{2\theta\log t}(\omega)\le H_{-(1-2\theta)\log t}(\omega)\quad\hbox{ if and only if }\quad \sup_{0\le z\le H_{2\theta\log t}(\omega)}
\left(-W(z,\omega)\right)\le (1-2\theta)\log t.
\end{align*}
For every $\om\in \Lambda_t^1$,
\begin{equation}\label{l4-2-5}
\sup_{z\in [0,1]}|W(z,\om)|= \sup_{z\in [0,1]}\left|W(z,\om)-W(0,\om)\right|\le \xi_{+,1}(\om)\le \xi_{+,\log^4 t}\left(\om\right)\le K_2\sqrt{\log\log t}.
\end{equation}
Choose $T_2\ge 2$ large enough so that $K_2\sqrt{\log\log t}\le 2\theta \log t$ for all $t\ge T_2$. Then,
$H_{2\theta \log t}(\om)>1$, so for every $\om \in \Lambda_t^1$ and $t>T_2$ (by choosing $T_2$ large if necessary),
$$
\int_0^{H_{\theta \log t}(\om)}e^{W(z,\om)}\,dz \le e^{\theta \log t}H_{\theta \log t}(\om)\le t^{\theta}\log^4 t\le t^{{3\theta}/{2}}
$$ and $$
\int_0^{H_{2\theta \log t}(\om)}e^{W(z,\om)}\,dz \ge \int_{H_{2\theta \log t}(\om)-1}^{H_{2\theta \log t}(\om)}
e^{W(z,\om)}\,dz\ge e^{2\theta \log t-K_2\sqrt{\log\log t}}\ge t^{{3\theta}/{2}}.
$$
Here in the second step of the inequality above we have used the fact that for every $\om \in \Lambda_t^1$ and $H_{2\theta \log t}(\om)-1\le z \le H_{2\theta \log t}(\om)$,
\begin{equation}\label{l4-2-3a}
\begin{split}
W(z,\om)&\ge W\left(H_{2\theta \log t}(\om),\om\right)-\left|W\left(H_{2\theta \log t}(\om),\om\right)-W(z,\om)\right|\\
&\ge 2\theta \log t-\xi_{+,H_{2\theta \log t}(\om)}\left(\om\right)\cdot\left|H_{2\theta \log t}(\om)-z\right|\\
&\ge 2\theta \log t-\xi_{+,\log^4 t}\left(\om\right)\ge 2\theta \log t-K_2\sqrt{\log\log t}.
\end{split}
\end{equation}
Therefore,
\begin{equation}\label{l4-2-3}
H_{\theta\log t}(\om)\le \delta_+(t^{{3\theta}/{2}},\om)\le
H_{2\theta\log t}(\om),\quad \om \in \Lambda_t^1.
\end{equation}

Let $z_0(\om)\in [0,H_{\theta\log t}(\om)]$ such that
\begin{align*}
-W\left(z_0(\om),\om\right)=r_1(\omega)=\sup_{0\le z \le H_{\theta\log t}(\omega)}\left(-W(z,\omega)\right).
\end{align*} Choosing $T_2$ large enough if necessary so that $K_2\sqrt{\log\log t}\le K_1\log\log t$ for all $t\ge T_2$. By this fact, \eqref{l4-2-5} and \eqref{l4-2-2}, we know
immediately that $z_0(\om)\ge 1$ when $t\ge T_2$. Hence, according to \eqref{l4-2-3}, for every $\om \in \Lambda_t^1$ and $t\ge T_2$,
\begin{equation}\label{l4-2-4}
\begin{split}
V_+(t^{3\theta/2},\om)&=\int_0^{\delta_+(t^{{3\theta}/{2}},\om)}e^{-W(z,\om)}\,dz
\ge \int_0^{H_{\theta\log t}(\om)}e^{-W(z,\om)}\,dz\\
&\ge \int_{z_0(\om)-1}^{z_0(\om)}e^{-W(z,\om)}\,dz
\ge e^{r_1(\om)-K_2\sqrt{\log\log t}}\ge e^{-K_2\sqrt{\log\log t}} \log^{K_1}t ,
\end{split}
\end{equation}
where the third inequality follows from the argument for \eqref{l4-2-3a}, and in the last inequality we used \eqref{l4-2-2}.
On the other hand, using \eqref{l4-2-2} and \eqref{l4-2-3} again, we derive that for every $\om \in \Lambda_t^1$ and $t\ge T_2$,
\begin{align*}
V_+(t^{{3\theta}/{2}},\om)&=\int_0^{\delta_+(t^{{3\theta}/{2}},\om)}e^{-W(z,\om)}\,dz\le
\int_0^{H_{2\theta \log t}(\om)}e^{-W(z,\om)}\,dz\\
&\le e^{\max\{r_1(\om),r_2(\om)\}}H_{2\theta \log t}(\om)\le t^{1-2\theta}\log^4 t.
\end{align*}
In particular,
choosing $T_2$ large if necessary, we find that for every $\om \in \Lambda_t^1$ and $t\ge T_2$,
$$
4t^{{3\theta}/{2}}V_+(t^{{3\theta}/{2}},\om)\le t^{1-{\theta}/{2}}\log^4 t \le t.
$$
This along with the definition of $R_+(t,\om)$ yields that that $$R_+(t,\om)\ge t^{{3\theta}/{2}}$$ for every
$\om\in \Lambda_t^1$ and $t\ge T_2$.

Now, combining this with \eqref{l2-2-1a} and \eqref{l4-2-4}, we have
\begin{align*}
p^X(t,0,0)&= p^Y(t,0,0)=p^Y\left(4R_+(t,\om)V_+\left(R_+(t,\om),\om\right),0,0\right)\\
&\le \frac{2}{V_+\left(R_+(t,\om),\om\right)}\le \frac{2}{V_+(t^{{3\theta}/{2}},\om)}\le \frac{c_1e^{K_2\sqrt{\log\log t}}}{\log^{K_1}t},\quad \om\in \Lambda_t^1,\ t\ge T_2.
\end{align*}
Therefore,
we can
find $K_1^*\ge 2$ large enough so that for every $K_1\ge K_1^*$,
\begin{align*}
p^X(t,0,0,\om)\le \frac{c_2}{\log^2 t},\quad \om\in \Lambda_t^1,\ t\ge T_2,
\end{align*}
which proves the first inequality in \eqref{l4-2-1} immediately.
\end{proof}

\begin{lemma}\label{l4-3}
Given any $K_1,K_2>0$, there exist  constants $T_3, C_5>0$ such that for any $t\ge T_3$,
\begin{equation}\label{l4-3-1}
\Ee\left[p^X(t,0,0)\I_{\Lambda_t^2\cap \tilde \Lambda_t^2}\right]\le
\frac{C_5(\log\log t)^{4+{1}/{(2\alpha)}}}{\log^2 t}.
\end{equation}
\end{lemma}
\begin{proof}
To prove the desired assertion. We will decompose $\Lambda_t^2$ by $\Lambda_t^2\subset \cup_{i=1}^3 \Lambda_t^{2i}$, where
\begin{align*}
&\Lambda_t^{21}:=\Big\{\om \in \Omega:  \sup_{0\le z\le H_{-(1-2\theta)\log t}(\om)}W(z,\om)=r_1(\om),\
\sup_{0\le z\le H_{(1-2\theta)r_1(\om)}(\om)}(-W(z,\om))=r_2(\om)\\
&\qquad\qquad\qquad \qquad{\rm with}\ \ K_3\log\log t\le r_1(\om) \le 2\theta \log t,\ 0\le r_2(\om)\le K_3\log\log t\Big\},\\
&\Lambda_t^{22}:=\Big\{\om \in \Omega: \sup_{0\le z\le H_{-(1-2\theta)\log t}(\om)}W(z,\om)=r_1(\om),\
\sup_{0\le z\le H_{(1-2\theta)r_1(\om)}(\om)}(-W(z,\om))=r_2(\om)\\
&\qquad\qquad\qquad \qquad{\rm with}\ \ K_3\log\log t\le r_1(\om) \le 2\theta \log t,\ K_3\log\log t< r_2(\om)\le (1-2\theta)\log t,\\
&\qquad\qquad \qquad\qquad \ H_{2\theta \log t}(\om)\le \log^4 t \text{\,\,and\,\,}\xi_{+,\log^4 t}(\om)\le K_2\sqrt{\log \log t}\Big\},\\
&\Lambda_t^{23}:=\Big\{\om\in \Omega: H_{-(1-2\theta)\log t}(\om)\le H_{K_3\log\log t}(\om)\Big\}
\end{align*}
with $K_3$ being  a positive constant to be determined later. Analogously, define
$\tilde \Lambda_t^{2i}$, $i=1,2,3$, as the same way as that for $\Lambda_t^{2i}$ with
$\tilde H_z(\om)$ and $\xi_{-,R}(\om)$ instead of $H_z(\om)$ and  $\xi_{+,R}(\om)$ respectively.

(i) Let $z_0(\om)\in [0,H_{-(1-2\theta)\log t}(\om)]$ be such that
$$W\left(z_0(\om),\om\right)=r_1(\om)=\sup_{0\le z\le H_{-(1-2\theta)\log t}(\om)}W(z,\om).$$
Then, according to the proof of \eqref{l4-2-5}, we know that $z_0(\om)>1$ (with $T_3$ large enough if necessary),  and so for every
$\om \in \Lambda_{t}^{22}$ and $t>T_3$
\begin{equation}\label{l4-3-2}
\begin{split}
\int_0^{H_{-(1-2\theta)\log t}(\om)}e^{W(z,\om)}\,dz&\ge \int_{z_0(\om)-1}^{z_0(\om)}e^{W(z,\om)}\,dz\ge e^{r_1(\om)-K_2\sqrt{\log\log t}}\ge e^{(1-\theta)r_1(\om)},
\end{split}
\end{equation}
where in the second inequality we have used
the same argument as that for \eqref{l4-2-3a}, and in the last inequality we used the fact that $r_1(\om)\ge K_3\log\log t$ for every $\om \in \Lambda_t^{22}.$
In particular, $$\delta_+(e^{(1-\theta)r_1(\om)},\om)\le H_{-(1-2\theta)\log t}(\om).$$ Therefore,
for every
$\om \in \Lambda_{t}^{22}$ and $t>T_3$,
\begin{align*}
V_+(e^{(1-\theta)r_1(\om)},\om)&=\int_0^{\delta_+\left(e^{(1-\theta)r_1(\om)},\om\right)}e^{-W(z,\om)}\,dz\le \int_0^{H_{-(1-2\theta)\log t}(\om)}e^{-W(z,\om)}\,dz\\
&\le e^{(1-2\theta)\log t}H_{-(1-2\theta)\log t}(\om)\le t^{(1-2\theta)}\log^4 t,
\end{align*}
where the last inequality is due to the fact that (thanks to the definition of $\Lambda_{t}^{22}$)
$$H_{-(1-2\theta)\log t}(\om)\le H_{2\theta \log t}(\om)\le \log^4 t,\quad \om \in \Lambda_{t}^{22}.$$
Thus, for every
$\om \in \Lambda_{t}^{22}$ and $t>T_3$ (by noting that $\theta\in (0,1/2)$ and by taking $T_3$ large enough if necessary),
\begin{align*}
4e^{(1-\theta)r_1(\om)}V_+ (e^{(1-\theta)r_1(\om)},\om )\le 4e^{2(1-\theta)\theta\log t}t^{(1-2\theta)}\log^4 t\le t=
4R_+(t,\om)V_+\left(R_+(t,\om),\om\right),
\end{align*}
which implies that
\begin{align}\label{l4-3-3}
R_+(t,\om)\ge e^{(1-\theta)r_1(\om)}.
\end{align}

On the other hand, for every $\om \in \Lambda_{t}^{22}$ and $t\ge T_3$,
\begin{align*}
\int_0^{H_{(1-2\theta)r_1(\om)}(\om)}e^{W(z,\om)}dz&\le e^{(1-2\theta)r_1(\om)}H_{(1-2\theta)r_1(\om)}(\om)
\le e^{(1-2\theta)r_1(\om)}\log^4 t\le e^{(1-\theta)r_1(\om)},
\end{align*}
where we have used the facts that
\begin{align*}
H_{(1-2\theta)r_1(\om)}(\om)\le H_{2\theta(1-2\theta)\log t}(\om)\le H_{2\theta \log t}(\om)\le \log^4 t,\quad \om \in \Lambda_{t}^{22},
\end{align*}
and, by taking $K_3$ large enough,
\begin{align*}
\theta r_1(\om)\ge \theta K_3\log\log t>4\log\log t,\quad \om \in \Lambda_{t}^{22}.
\end{align*}
This implies immediately that $\delta_+\left(e^{(1-\theta)r_1(\om)},\om\right)\ge H_{(1-2\theta)r_1(\om)}(\om)$, and so
\begin{equation}\label{e:oo}\begin{split}
V_+ (e^{(1-\theta)r_1(\om)},\om )&=\int_0^{\delta_+ (e^{(1-\theta)r_1(\om)},\om )}e^{-W(z,\om)}\,dz\\
&\ge \int_0^{H_{(1-2\theta)r_1(\om)}(\om)}e^{-W(z,\om)}\,dz\ge e^{(1-\theta)r_2(\om)},\quad \om\in \Lambda_t^{22},\ t\ge T_3,
\end{split}\end{equation}
Here the last inequality follows from (by taking $T_3$ large enough if necessary)
\begin{align*}
\int_0^{H_{(1-2\theta)r_1(\om)}(\om)}e^{-W(z,\om)}dz&\ge \int_{z_1(\om)-1}^{z_1(\om)}e^{-W(z,\om)}dz\ge
e^{r_2(\om)-K_2\sqrt{\log\log t}}\ge e^{(1-\theta)r_2(\om)},\quad \om \in \Lambda_t^{22},\ t\ge T_3,
\end{align*}
where $z_1(\om)\in [0,H_{(1-2\theta)r_1(\om)}(\om)]$ satisfies
\begin{align*}
-W\left(z_1(\om),\om\right)=r_2(\om)=\sup_{0\le z\le H_{(1-2\theta)r_1(\om)}(\om)}(-W(z,\om))
\end{align*}
and $z_1(\om)\ge1$ that can be proved by exactly the same way as that of
\eqref{l4-2-5}.

By \eqref{l2-2-1a}, \eqref{l4-3-3} and \eqref{e:oo}, we deduce that for every $\om \in \Lambda_t^{22}$ and $t\ge T_3$,
\begin{align*}
p^X(t,0,0,\om)&=p^Y(t,0,0,\om)=p^Y\left(4R_+(t,\om)V_+\left(R_+(t,\om),\om\right),0,0,\om\right)\\
&\le \frac{2}{V_+\left(R_+(t,\om),\om\right)}\le c_1e^{-(1-\theta)r_2(\om)}\le \frac{c_1}{\log^{(1-\theta)K_3}t}.
\end{align*}
In particular, taking $K_3\ge1$ large enough so that $K_3(1-\theta)>2$, we have
\begin{equation}\label{l4-3-4}
p^X(t,0,0,\om)\le \frac{c_2}{\log^2 t},\quad \om\in \Lambda_t^{22},\ t\ge T_3.
\end{equation}
Similarly, we can obtain
\begin{align*}
p^X(t,0,0,\om)\le \frac{c_2}{\log^2 t},\quad \om\in \tilde \Lambda_t^{22},\ t\ge T_3.
\end{align*}

(ii)  According to \cite[(2.1.2), p.\ 204]{BS},
$$
\Pp_x\left(\sup_{0\le s\le H_z}W(s)\in dy\right)=\frac{x-z}{(y-z)^2}\,dy,\quad z\le x\le y,
$$
where $\Pp_x(\cdot)=\Pp(\cdot|W(0)=x)$ denotes the conditional probability given the event that $W(0,\om)=x$.
Therefore, by the strong Markov property and the fact that we assume $W(0)=0$,
for every $K_3\log \log t \le y_1 \le 2\theta \log t$ and $0\le y_2\le K_3\log\log t$,
\begin{align*}
& \Pp\left(\sup_{0\le s\le H_{-(1-2\theta)\log t}}W(s)\in dy_1, \sup_{0\le s\le H_{(1-2\theta)y_1}}(-W(s))\in dy_2\right)\\
&=\Pp\left(\sup_{0\le s\le H_{(1-2\theta)y_1}}(-W(s))\in dy_2\cdot \Pp_{W(H_{(1-2\theta)y_1})}\left(\sup_{0\le s\le H_{-(1-2\theta)\log t}}W(s)\in dy_1\right)\right)\\
&=\frac{(1-2\theta)y_1}{\left((1-2\theta)y_1+y_2\right)^2}\cdot \frac{(1-2\theta)\log t+(1-2\theta)y_1}{\left(y_1+(1-2\theta)\log t\right)^2}\,dy_1\,dy_2.
\end{align*}
This immediately yields that
\begin{equation}\label{l4-3-5}
\begin{split}
\Pp\left(\Lambda_t^{21}\right)&\le \int_{K_3\log\log t}^{2\theta \log t}\int_0^{K_3\log\log t}
\frac{(1-2\theta)y_1}{\left((1-2\theta)y_1+y_2\right)^2}\cdot \frac{(1-2\theta)\log t+(1-2\theta)y_1}{\left(y_1+(1-2\theta)\log t\right)^2}\,dy_2\,dy_1\\
&\le \frac{c_3\log \log t
}{\log t}\int_{K_3\log\log t}^{2\theta \log t}\frac{1}{r_1}\,dr_1\le \frac{c_4 (\log \log t)^2 }{\log t}
\end{split}
\end{equation}
Similarly, we have
\begin{align*}
\Pp (\tilde \Lambda_t^{21} )\le \frac{c_4 (\log \log t)^2 }{\log t}.
\end{align*}

(iii) According to \cite[(2.2.2), p.\ 204]{BS}, for every $t\ge T_3$,
\begin{equation}\label{l4-3-6}
\begin{split}
\Pp (\Lambda_t^{23} )=
\Pp (\tilde \Lambda_t^{23} )&=
\Pp\left(\inf_{0\le z\le H_{K_3\log\log t}}W(z)\le -(1-2\theta)\log t\right)\\
&=\frac{K_3\log\log t}{K_3\log\log t+(1-2\theta)\log t}\le \frac{c_5\log\log t}{\log t}.
\end{split}
\end{equation}

Note that $\Lambda_t^{2i}$ and $\tilde \Lambda_t^{2i}$, $i=1,2,3$,
are independent with each other. Thus, for every $t\ge T_2$,
\begin{align*}
\Pp (\Lambda_t^{2i}\cap \tilde \Lambda_t^{2j} )=
\Pp (\Lambda_t^{2i} )\cdot\Pp (\tilde \Lambda_t^{2j} )
\le \frac{c_6(\log \log t)^4}{\log^2 t},\quad i=1,3,j=1,3.
\end{align*}
This along with \eqref{l4-1-0a} yields that
\begin{align*}
\Ee\left[p^X(t,0,0)\I_{\Lambda_{t}^{2i}\cap \tilde \Lambda_t^{2j}}\right]\le
\frac{c_7(\log \log t)^{4+{1}/{(2\alpha)}}}{\log^2 t},\quad i=1,3,j=1,3,\ t\ge T_3.
\end{align*}
Therefore, putting the inequality above and \eqref{l4-3-4} together, we find that
\begin{align*}
  \Ee\left[p^X(t,0,0)\I_{\Lambda_t^2 \cap \tilde \Lambda_t^2}\right]
&\le
\Ee\left[p^X(t,0,0)\I_{\left(\cup_{i=1}^3\Lambda_t^{2i}\right)\cap\left(\cup_{j=1}^3\tilde \Lambda_t^{2j}\right)}\right]\\
&\le 2\Ee\left[p^X(t,0,0)\I_{\Lambda_t^{22}}\right]+2\Ee\left[p^X(t,0,0)\I_{\tilde \Lambda_t^{22}}\right]
+\sum_{i=1,3}\sum_{j=1,3}\Ee\left[p^X(t,0,0)\I_{\Lambda_{t}^{2i}\cap \tilde \Lambda_t^{2j}}\right]\\
&\le \frac{c_8(\log \log t)^{4+{1}/{(2\alpha)}}}{\log^2 t}.
\end{align*}
The proof is complete. \end{proof}

\begin{lemma}\label{l4-4}
There exist positive constants $K_2^*$,  $T_4$ and $C_6$ such that if $K_2>K_2^*$ in the definitions of $\Lambda_t^4$ and
$\tilde \Lambda_t^4$, then for $t\ge T_4$,
\begin{equation}\label{l4-4-1}
\Ee\left[p^X(t,0,0)\I_{\Lambda_t^i\cap \tilde \Lambda_t^j}\right]\le
\frac{C_6(\log\log t)^{4+{1}/{(2\alpha)}}}{\log^2 t},
\quad i=2,3,4,5,j=2,3,4,5.
\end{equation}
\end{lemma}
\begin{proof}
(i) According to \cite[(2.0.2), p.\ 204]{BS},  there is $T_4>0$ so that
for every $t\ge T_4$,
\begin{align*}
\Pp(\Lambda_t^3)=\Pp(\tilde \Lambda_t^3)&=
\Pp(H_{2\theta\log t}> \log^4 t)\\
&=\frac{2\theta\log t}{\sqrt{2\pi}}\int_{\log^4 t}^\infty
\frac1{y^{3/2}}\exp\left(-\frac{4\theta^2\log^2t}{2y}\right)\,dy\\
&=\frac{2\theta}{\sqrt{2\pi}}\int_{\log^2 t}^\infty\frac1{y^{3/2}}\exp\left(-\frac{2\theta^2}{y}\right)\,dy\leq \frac{c_1}{\log t}.
\end{align*}
By \cite[(2.2.2), p.\ 204]{BS}, for every $t\ge T_4$
(by choosing $T_4$ large if necessary),
\begin{align*}
\Pp (\Lambda_t^{5} )=
\Pp (\tilde \Lambda_t^{5} )&=
\Pp\left(\inf_{0\le z\le H_{-K_1\log\log t}}(-W(z))\le -\theta\log t\right)\\
&=\frac{K_1\log\log t}{K_1\log\log t+\theta\log t}\le \frac{c_2\log\log t}{\log t}.
\end{align*}
Meanwhile, according to the argument of \eqref{ieqR}, we know that for every $n\ge 0$ and $R\ge 1$,
\begin{align*}
\Pp\left(\sup_{n\leq z_1,z_2\leq n+2:\,|z_1-z_2|\leq1}\frac{|W(z_1)-W(z_2)|}{|z_1-z_2|^\alpha}>R\right)
\le c_3e^{-c_4R^2}.
\end{align*}
Hence we can find a $K_2^*\ge 1$ so that for every $K_2\ge K_2^*$, it holds that
\begin{equation}\label{l4-4-2}
\begin{split}
\Pp (\Lambda_t^{4} )=
\Pp (\tilde \Lambda_t^{4} )&=
\Pp (\xi_{+,\log^4t}> K_2\sqrt{\log\log t} )
\\& \leq \sum_{n=0}^{\lceil\log^4t\rceil+1}\Pp\left(\sup_{n\leq z_1,z_2\leq n+2: |z_1-z_2|\leq1}\frac{|W(z_1)-W(z_2)|}{|z_1-z_2|^\alpha}> K_2\sqrt{\log\log t}\right)
\\&\leq c_5\log^4t\exp (-c_4K_2^2\log\log t)
\le \frac{c_5}{\log^{c_4K_2^2-4} t}
\leq\frac{c_6}{\log t}.
\end{split}
\end{equation}
Putting all the estimates above together and using the fact that $\Lambda_t^i$ and
$\tilde \Lambda_t^j$, $i=3,4,5$ and $j=3,4,5$, are independent with each other, we derive
$$
\Pp (\Lambda_t^{i}\cap \tilde \Lambda_t^{j})=
\Pp (\Lambda_t^{i})\cdot\Pp (\tilde \Lambda_t^{j})
\le \frac{c_7(\log \log t)^2}{\log^2 t},\quad i=3,4,5,\ j=3,4,5.
$$
Therefore, combining this with Lemma \ref{small}, we get
\eqref{l4-4-1} for every $i=3,4,5$ and $j=3,4,5$.

(ii) According to the estimates above and \eqref{l4-3-1}, it suffices to show that
\eqref{l4-4-1} for every $i=2$ and $j=3,4,5$, and for every $i=3,4,5$ and $j=2$. Here, we only prove
the case that $i=2$ and $j=3,4,5$, and the other case can be shown by exactly the same way.

As in Lemma \ref{l4-3}, we decompose $\Lambda_t^2$ as $\Lambda_t^2=\cup_{i=1}^3\Lambda_t^{2i}$. By
\eqref{l4-3-4}, \eqref{l4-3-5} and \eqref{l4-3-6} as well as
the independence of $\Lambda_t^{2i}$ and $\tilde \Lambda_t^j$,
\begin{align*}
\Ee\left[p^X(t,0,0)\I_{\Lambda_t^{22}}\right]\le \frac{c_8}{\log^2 t}
\end{align*} and
\begin{align*}
\Pp (\Lambda_t^{2i}\cap \tilde \Lambda_t^j )&
=\Pp (\Lambda_t^{2i} )\cdot\Pp (\tilde \Lambda_t^j )\le
\frac{c_8(\log\log t)^4}{\log^2 t},\quad i=1,3,\ j=3,4,5.
\end{align*}
This along with \eqref{l4-1-0a} yields that for every $t\ge T_4$
\begin{align*}
\Ee\left[p^X(t,0,0)\I_{\Lambda_t^{2i}\cap \tilde  \Lambda_t^j}\right]\le
\frac{c_9(\log\log t)^{4+{1}/{(2\alpha)}}}{\log^2 t},\quad i=1,3,\ j=3,4,5.
\end{align*}
Hence, combining with all the estimates above yields that
\begin{align*}
\Ee\left[p^X(t,0,0)\I_{\Lambda_t^2\cap \tilde \Lambda_t^j}\right]&
\le \Ee\left[p^X(t,0,0)\I_{\Lambda_t^{22}}\right]+\sum_{i=1,3}\Ee\left[p^X(t,0,0)\I_{\Lambda_t^{2i}\cap \tilde  \Lambda_t^j}\right]\\
&\le \frac{c_{10}(\log\log t)^{4+1/{(2\alpha)}}}{\log^2 t},\quad \ j=3,4,5.
\end{align*}
The proof is complete.
\end{proof}

With all the estimates above, we are in a position to present the

\begin{proof}[Proof of Proposition $\ref{T:UP}$]
Note that $\Omega=\cup_{i=1}^5 \Lambda_t^i=\cup_{i=1}^5 \tilde \Lambda_t^i$. By
\eqref{l4-2-1} and \eqref{l4-4-1}, we can choose suitable positive constants $K_1,K_2$
in the definition of $\Lambda_t^i$, $\tilde \Lambda_t^i$, $1\le i \le 5$, such that for every $t\ge T_1$ with $T_1:=\max\{T_2,T_3,T_4\}$,
\begin{align*}
\Ee\left[p(t,0,0)\right]&=\Ee\left[p(t,0,0)
\I_{\left(\cup_{i=1}^5\Lambda_t^i\right)\cap \left(\cup_{j=1}^5 \tilde \Lambda_t^j\right)}\right]\\
&\le 2\Ee\left[p(t,0,0)\I_{\Lambda_t^1}\right]
+2\Ee\left[p(t,0,0)\I_{\tilde \Lambda_t^1}\right]+
\sum_{i=2}^5\sum_{j=2}^5\Ee\left[p(t,0,0)\I_{\Lambda_t^i \cap \tilde \Lambda_t^j}\right]\\
&\le \frac{c_{1}(\log\log t)^{4+1/{(2\alpha)}}}{\log^2 t},
\end{align*}
so the proof is finished.
\end{proof}

\subsection{Lower bound}
In this subsection, we prove the following proposition.
\begin{proposition}\label{thm:lower} There are constants $C_1, T_1>0$ so that for all $t\ge T_1$,
$$\Ee[p(t,0,0)]\geq\frac{C_1}{(\log t)^2(\log\log t)^{11}}.$$
\end{proposition}

Firstly, we define the following subsets of $\Omega$,
\begin{align*}
\Gamma_t =&\{\om\in \Omega:
H_{-1,2\log t}(\om)< H_{-1}(\om)\}=\{\om\in \Omega:
H_{-1,2\log t}(\om)= H_{2\log t}(\om)\},\\
 \Theta_t^1=&\{\om\in \Omega: H_{-1}(\om)\leq K_1\log^2t\}, \\
 \Theta_t^2=&\{\om\in\Omega: \xi_{+,K_1\log^2t}(\om)\leq K_2\sqrt{\log \log t}\},\\
\Theta_t^3=&\left\{\om\in \Omega:\int_0^{H_{-1,2\log t}(\om)}
\I_{[-1,2 \log\log t]}\left(W(z,\om)\right)\,dz \le K_3(\log\log t)^3\right\},\end{align*}
 and
$$\Theta_t^4=\left\{\om\in \Omega: \int_0^{H_{{\log t}/{2}}(\om)}
\I_{[-1,0]}\left(W(z,\om)\right)\,dz\ge \frac{1}{K_4\log\log t}\right\},$$
where the positive constants $K_i$, $1\le i\le 4$ will be fixed later.
Similarly, we can define $\tilde \Gamma_t$ and $\tilde \Theta_t^i$, $1\le i\le 4,$ by using $\tilde{W}(z,\om)=W(-z,\om)$ in place of $W(z,\om)$.
Finally, set
$$\Omega_t:=\Gamma_t\cap \left(\cap_{i=1}^4 \Theta_t^i\right)\cap \tilde \Gamma_t\cap (\cap_{i=1}^5 \tilde \Theta_t^i).$$

\bl There are $K_1,K_2,K_3,K_4$ large enough and $T_2>0$ so that
\begin{equation}\label{l4-5-1}
\Pp(\Omega_t)\ge  \frac{1}{4\log^2 t},\quad t\ge T_2.
\end{equation}
\el
\begin{proof} Throughout the proof, we always assume that $t\ge T_2$ for some $T_2$ large enough.
By \cite[(3.0.4)(b), p.\ 218]{BS}, we know that
$$\Pp(\Gamma_t)= \frac{1}{1+2\log t}.$$

According to \cite[(2.0.2), p.\ 204]{BS}, we can take $K_1$ large enough so that
\begin{align*}
\Pp\left((\Theta_t^1)^c\right)&=
\Pp\left(H_{-1}>K_1\log^2 t\right)\\
&=\frac{1}{\sqrt{2\pi}}\int_{K_1\log ^2 t}^\infty \frac{1}{s^{3/2}}\exp\left(-\frac{1}{2s}\right)\,d s\le \frac{ c_1}{ K_1^{1/2}\log  t}\le \frac{1}{8\log t}.
\end{align*}

Meanwhile, following the argument of \eqref{l4-4-2} and taking $K_2$ large enough, we have
\begin{align*}
\Pp\left((\Theta_t^2)^c\right)\le c_2(K_1\log ^2 t)\exp\left(-c_3 K_2^2\log\log t\right)\le
\frac{1}{8\log t},\quad t\ge T_2.
\end{align*}

To consider $\Theta_t^3$, we recall from \cite[(3.5.5)(b), p.\ 224]{BS} that, for any $\lambda>0$,
\begin{align*}
&\Ee\left[\exp\left(-\lambda\int_0^{H_{-1,2\log t}} \I_{[-1,2\log\log t]}(W(z))\,dz\right);
W\left(H_{-1,2\log t}\right)=2\log t\right]\\
&=\frac{\textrm{sh}(\sqrt{2\lambda})}
{\textrm{sh}(\sqrt{2\lambda}(1+2\,\log\log t))+\sqrt{2\lambda}(2\log t-2\, \log\log t)\textrm{ ch}(\sqrt{2\lambda}(1+2\, \log\log t))}\\
&=:Q(\lambda),
\end{align*} where
$$\textrm{sh}(x)=\frac{1}{2}(e^x-e^{-x}),\quad \textrm{ch}(x)=\frac{1}{2}(e^x+e^{-x}).$$
Thus, thanks to \cite[(3.0.4)(b), p.\ 218]{BS} again, we find that
\begin{align*}
&\Ee\left[ \int_0^{H_{-1,2\log t}} \I_{[-1,2\, \log\log t]}(W(z))\,dz; W\left(H_{-1,2\log t}\right)=2\log t\right]\\
&=\lim_{\lambda\downarrow 0} \frac{\Pp\left(W\left(H_{-1,2\log t}\right)=2\log t\right)-Q(\lambda)}{\lambda}\\
&=\lim_{\lambda\downarrow 0} \frac{(1+2\log t)^{-1}-Q(\lambda)}{\lambda}\le\frac{ c_{4}(\log \log t)^3}{\log t},
\end{align*}
which yields that
\begin{align*}
\Pp\left(\Gamma_t\cap (\Theta_t^3)^c\right)&=\Pp\left(\int_0^{H_{-1,2\log t}} \I_{[-1,2\, \log\log t]}(W(z))\,dz>K_3(\log\log t)^3;W\left(H_{-1,2\log t}\right)=2\log t \right)\\
&\le \frac{\Ee\left[ \displaystyle\int_0^{H_{-1,2\log t}} \I_{[-1,2\, \log\log t]}(W(z))\,dz; W\left(H_{-1,2\log t}\right)=2\log t\right]}{K_3(\log\log t)^3}\\
&\le \frac{c_{5}}{K_3\log t}.
\end{align*}
In particular, by choosing $K_3$ large enough, we arrive at
$$\Pp\left(\Gamma_t\cap (\Theta_t^3)^c\right)\le \frac{1}{8\log t},\quad t\ge T_2.$$

Furthermore, let $\theta\in (0,1/2)$ be fixed later. It holds that
\begin{align*}
\Pp_0\left((\Theta_t^4)^c\right)
&\le \Pp \left(\int_0^{H_{\frac{\log t}{2}}}\I_{(-\infty,0)}(W(z))\,dz\le \frac{1}{K_4\log\log t}\right)\\
&\le \Pp \left(\int_0^{\frac{\theta\log^2 t}{\log\log t}}\I_{(-\infty,0)}(W(z))\,dz\le \frac{1}{K_4\log\log t}\right)+
\Pp\left(H_{\frac{\log t}{2}}\le \frac{\theta\log^2 t}{\log\log t}\right).
\end{align*}
Applying  \cite[(2.0.2), Page\ 204]{BS} again, we derive
\begin{align*}
\Pp\left(H_{\frac{\log t}{2}}\le \frac{\theta\log^2 t}{\log\log t}\right)&=\frac{\log t}{2\sqrt{2\pi}}\int_0^{\frac{\theta\log^2 t}{\log\log t}}\frac{1}{s^{3/2}}
\exp\left(-\frac{\log^2 t}{4s}\right)ds\\
&=\frac{1}{2\sqrt{2\pi}}\int_0^{\frac{\theta}{\log\log t}}\frac{1}{s^{3/2}}
\exp\left(-\frac{1}{4s}\right)ds\\
&\le \frac{c_6}{\log^{\frac{1}{8\theta}}t},\quad t\ge T_2.
\end{align*}
Hence, choosing $\theta\in (0,1/2)$ small enough, we obtain
\begin{align*}
\Pp\left(H_{\frac{\log t}{2}}\le \frac{\theta\log^2 t}{\log\log t}\right)\le \frac{1}{16\log t},\quad t\ge T_2.
\end{align*}
On the other hand, it is well known that
$\int_0^1 \I_{(-\infty,0)}(W(z))\,dz$ satisfies the arcsin law, i.e., for any $a>0$,
$$\Pp\left(\int_0^1\I_{(-\infty,0)}(W(z))\,dz\le a\right)=\sqrt{\frac{2}{\pi}}{\arcsin \sqrt{a}}.$$
Note that, by the scaling invariance property of Brownian motion, for every $T>0$,
$\int_0^T \I_{(-\infty, 0)}(W(z))\,dz$ and $T\int_0^11_{(-\infty, 0)}(W(z))\,dz$ enjoy the same law.
Hence, for $\theta$ small enough fixed above, we can choose $K_4$ large enough so that
\begin{align*}
\Pp\left(\int_0^{\frac{\theta\log ^2 t}{\log\log t} }\I_{(-\infty,0)}(W(z))\,dz\le \frac{1}{K_4\log\log t}\right)=&
\Pp\left(\int_0^{1}\I_{(-\infty,0)}(W(z))\,dz\le \frac{1}{\theta K_4\log^2t}\right)\\
\le &\sqrt{\frac{2}{\pi}}\arcsin \sqrt{\frac{1}{\theta K_4\log ^2 t}} \le \frac{1}{16\log t}.
\end{align*}
Thus,  $$ \Pp\left((\Theta_t^4)^c\right)\le \frac{1}{8\log t},\quad t\ge T_2.$$

Putting all the estimates together, we arrive at
$$\Pp\left(\Gamma_t\cap (\cap_{i=1}^4\Theta_t^i)\right)\ge \Pp(\Gamma_t)-\sum_{i=1}^4 \Pp\left(\Gamma_t\cap (\Theta_t^i)^c\right)\ge
\frac{1}{2\log t},\quad t\ge T_2.$$ This, together with the independence of
$\Gamma_t\cap (\cap_{i=1}^5\Theta_t^i)$ and $\tilde \Gamma_t\cap (\cap_{i=1}^4\tilde \Theta_t^i)$ as well as the fact that both sets have the same probability, gives us that
$$\Pp(\Omega_t)=\Pp\left(\Gamma_t\cap (\cap_{i=1}^4\Theta_t^i)\right)
\cdot \Pp (\tilde \Gamma_t\cap (\cap_{i=1}^4\tilde \Theta_t^i) )\ge \frac{1}{4\log^2 t},\quad t\ge T_2.$$
The proof is finished.
\end{proof}

\bl There are  constants $C_2, T_3>0$ such that for all $t\ge T_3$ and $\om\in \Omega_t$,
\begin{align}\label{l4-6-1}
p(t,0,0,\om)\ge \frac{C_2}{(\log\log t)^{11}}.
\end{align}
\el
\begin{proof}
Again in the proof, we will choose $T_3>0$ large enough and consider $t\ge T_3$. For any $t>0$ and $\om\in \Omega$, define $R(t,\om)$ by the unique real number so that
(which is slightly different from that in \eqref{e4-1}),
$$t=C_2R(t,\om)V\left(R(t,\om),\om\right),$$ where
$C_2$ is the constant given in \eqref{l2-4-0}. Below, we will take positive constants
$\kappa_1$ and $\kappa_2$, whose exact values will be determined later, so that
for all $t\ge T_3$ and $\om\in \Omega_t$,
\begin{align*}
\int_0^{H_{\frac{\log t}{2}}(\om)} e^{W(z,\om)}\,dz\le& e^{\frac{\log t}{2}}H_{\frac{\log t}{2}}(\om)\le t^{{1}/{2}}H_{2\log t}(\om)\\
\le & t^{{1}/{2}}H_{-1}(\om)\le  t^{{1}/{2}}K_1\log ^2 t
\le \frac{\kappa_2 t}{(\log\log t)^3}
\le \kappa_1 t\log\log t.
\end{align*}
In particular, for all $t\ge T_3$ and $\om\in \Omega_t,$
\begin{align*}
\delta_+\left(\kappa_1t\log\log t,\om\right)\ge \delta_+\left(\frac{\kappa_2 t}{(\log\log t)^3},\om\right)\ge H_{\frac{\log t}{2}}(\om),
\end{align*}
and  (noting that $\om\in \Omega_t\subset \Theta_t^4$),
\begin{equation}\label{e:proof-11}
\begin{split}
V\left(\kappa_1t\log\log t,\om\right)&\ge V\left(\frac{\kappa_2 t}{(\log\log t)^3},\om\right)
\ge \int_0^{\delta_+\left(\frac{\kappa_2 t}{(\log\log t)^3},\om\right)} e^{-W(z,\om)}\,dz\\
&\ge \int_0^{H_{\frac{\log t}{2}}(\om)} e^{-W(z,\om)}\,dz
\ge \int_0^{H_{\frac{\log t}{2}}(\om)}\I_{[-1,0]}(W(z,\om))\,dz\\
&\ge \frac{1}{K_4\log\log t}.
\end{split}
\end{equation}
Therefore, choosing $\kappa_1\ge C_2^{-1}K_4$ (with $C_2$ being the constant in the definition of $R(t,\om)$), we have
$$\kappa_1t\log\log t\cdot V\left(\kappa_1t\log\log t,\om\right)\ge \frac{\kappa_1t}{K_4}\ge C_2^{-1}t=R(t,\om)V(R(t,\om),\om),$$
and so
\begin{equation}\label{e:low1}
R(t,\om)\le \kappa_1t\log\log t,\quad  t\ge T_3,\om \in \Omega_t.
\end{equation}

On the other hand, for every $t\ge T_3$ and $\om \in \Omega_t$
(by noting that $\om\in\Omega_t\subset \Theta_t^2$),
\begin{align*}
\int_0^{H_{2\log t}(\om)}e^{W(z,\om)}\,dz&\ge \int_{H_{2\log t}(\om)-1}^{H_{2\log t}(\om)}e^{W(z,\om)}\,dz
\ge e^{2\log t-K_2\sqrt{\log\log t}}\\
&\ge 2\kappa_1 t(\log\log t)\ge \frac{\kappa_2 t}{(\log\log t)^3}.
\end{align*}
Here we used the facts that
$
1<H_{2\log t}(\om)\le K_1\log^2t$ and $W(z,\om)\ge e^{2\log t-K_2\sqrt{\log\log t}}
$ for all $H_{2\log t}(\om)-1\le z\le H_{2\log t}(\om)$,
which can be verified as the same way as these for \eqref{l4-2-5} and \eqref{l4-2-3a} and by taking $T_3$ large enough if necessary.
Hence,
$$\delta_+\left(\frac{\kappa_2t}{(\log \log t)^3},\om\right)
\le \delta_+\left(2\kappa_1 t\log\log t,\om\right)\le H_{2\log t}(\om),\quad t\ge T_3,\ \om \in \Omega_t.$$
Thus, by  $\om\in \Omega_t\subset \Gamma_t \cap \Theta_t^3$, we find that (by taking $T_3$ large enough if necessary)
\begin{align*}
&\int_0^{H_{2\log t}(\om)}e^{-W(z,\om)}\,dz= \int_0^{H_{-1,2\log t}(\om)}e^{-W(z,\om)}\,dz\\
&\le e\int_0^{H_{-1,2\log t}(\om)}\I_{[-1,2\log\log t]}(W(z,\om))\,d z+e^{-2\log\log t}\int_0^{H_{-1,2\log t}(\om)}
\I_{[2\,\log\log t,2\log t]}(W(z,\om)) \,dz\\
&\le eK_3(\log\log t)^3+(\log t)^{-2\,} H_{-1}(\om)\le  eK_3(\log\log t)^3+K_1(\log t)^2(\log t)^{-2\,}\\
&\le 2e K_3 (\log\log t)^3.
\end{align*}
Hence,
$$\int^{\delta_+\left(2\kappa_1 t\log\log t,\om\right)}_0e^{-W(z,\om)}\,dz
\le \int_0^{H_{2\log t}(\om)}e^{-W(z,\om)}\,dz \le 2 eK_3(\log\log t)^3.$$
Similarly, it holds that
$$\int_{\delta_-\left(2\kappa_1t\log\log t,\om\right)}^0e^{-W(z,\om)}\,dz\le 2 eK_3(\log\log t)^3,\quad t\ge T_3,\ \om \in \Omega_t.$$
So,
\begin{equation}
\label{e:proof-12}
V\left(\frac{\kappa_2t}{(\log \log t)^3},\om\right)\le
V\left(2\kappa_1t\log\log t,\om\right)\le 4eK_3 (\log\log t)^3,\quad t\ge T_3,\ \om \in \Omega_t.
\end{equation}
Thus, by taking $\kappa_2$ small enough such that $4e\kappa_2K_3\le C_2^{-1}$, we derive
$$\frac{\kappa_2t}{(\log \log t)^3}V\left(\frac{\kappa_2t}{(\log \log t)^3},\om\right)\le 4e\kappa_2K_3t\le
C_2^{-1}t=R(t,\om)V(R(t,\om),\om),$$
which implies that
\begin{equation}\label{e:low2}
R(t,\om)\ge \frac{\kappa_2t}{(\log \log t)^3},\quad t\ge T_3,\ \om\in \Omega_t.
\end{equation}

Combining \eqref{e:low1} with \eqref{e:low2}, we have
$$ \frac{\kappa_2t}{(\log \log t)^3}\le R(t,\om)\le  \kappa_1t\log\log t,\quad t\ge T_3,\ \om\in \Omega_t.$$
This along with \eqref{l2-4-0} yields that
\begin{align*}
p^X(t,0,0,\om)&=p^Y(t,0,0,\om)=p^Y\left(C_2R(t,\om)V\left(R(t,\om),\om\right),0,0,\om\right)\\
&\ge\frac{C_2^2 V(R(t,\om),\om)^2}{4C_1^2V(2R(t,\om),\om)^3}
\ge \frac{C_2^2 V\left(\frac{\kappa_2t}{(\log \log t)^3},\om\right)^2}{4C_1^2V\left(2\kappa_1t\log\log t,\om\right)^3},
\quad t\ge T_3,\ \om\in \Omega_t.
\end{align*}
where $C_1$ and $C_2$ are the constants given in \eqref{l2-4-0}.

Recall that, according to \eqref{e:proof-11} and \eqref{e:proof-12},
for every $t\ge T_3$ and $\om \in \Omega_t$,
$$V\left(\frac{\kappa_2t}{(\log \log t)^3},\om\right)\ge \frac{c_1}{\log\log t},$$
and
$$V\left({2\kappa_1t}{(\log \log t)},\om\right)\le  c_2(\log\log t)^3.$$
Putting them into the inequality above proves the desired assertion \eqref{l4-6-1}.
\end{proof}

\begin{proof}[Proof of Proposition $\ref{thm:lower}$]
According to \eqref{l4-5-1} and \eqref{l4-6-1}, for every $t\ge T_1:=\max\{T_2,T_3\}$,
$$
\Ee\left[p(t,0,0)\right]\ge \Ee\left[p(t,0,0)\I_{\Omega_t}\right]\ge \frac{c_1}{(\log\log t)^{11}}\cdot \Pp\left(\Omega_t\right)
\ge \frac{c_1}{4(\log t)^2(\log\log t)^{11}}.
$$
So the proof is finished.
\end{proof}

\ \

Finally, the assertion of Theorem \ref{thm22} for the case that $x=0$ is a consequenece of Propositions \ref{T:UP} and \ref{thm:lower}, and one can follow the similar arguments to deal with general $x\in \R$.

\ \

\noindent {\bf Acknowledgements.}\,\,  The authors would like to thank Professors Rongchan Zhu and Xiangchan Zhu for
introducing this problem to us, as well as
helpful discussions on topics related to this work.
The research of Xin Chen is supported by the National Natural Science Foundation
of China (No. 12122111). The research of Jian Wang is supported by the National Key R\&D Program of China (2022YFA1006003) and the  National Natural Science Foundation of China (Nos. 12225104 and 12531007).

\end{document}